\documentclass[final,11pt]{siamart251216}

\usepackage[english]{babel}
\usepackage{framed,multirow}
\usepackage{hyperref} 
\hypersetup{
    colorlinks=true,
    linkcolor=black,      
    citecolor=black,      
    filecolor=black,      
    urlcolor=black        
}
\usepackage{algpseudocode}
\usepackage{amssymb,amsmath,bm,amsfonts}
\usepackage{algorithm}
\usepackage{algcompatible}

\DeclareMathAlphabet{\mathbbm}{U}{bbm}{m}{n}

\usepackage{dsfont}
\usepackage{latexsym}
\usepackage{subcaption}
\usepackage{url}
\usepackage{xcolor}
\usepackage{mathtools}
\usepackage{cleveref}
\usepackage{showkeys}
\usepackage{multirow}
\usepackage{booktabs}
\usepackage{enumitem}  
\usepackage{array}
\usepackage{cancel}
\usepackage{algorithm}
\usepackage{algpseudocode}
\usepackage{mathabx} 

\newcommand{\A}{\mathcal{A}} 
\newcommand{\B}{\mathcal{B}}

\usepackage{tikz}
\usetikzlibrary{arrows.meta, positioning, calc, shapes, fit, backgrounds}
\tikzstyle{terminator} = [rectangle, draw, text centered, rounded corners]
\tikzstyle{process} = [rectangle, draw, text centered]
\tikzstyle{decision} = [diamond, draw, text centered]
\tikzstyle{data}=[trapezium, draw, text centered, trapezium left angle=60, trapezium right angle=120]
\tikzstyle{connector} = [draw, -{latex[length=2mm]}, thick, blue]
\tikzstyle{connector2} = [draw, -{latex[length=2mm]}, thick, green]
\pgfdeclarelayer{background}
\pgfsetlayers{background,main}

\newsiamremark{remark}{Remark}

\numberwithin{equation}{section}

\theoremstyle{definition}

\theoremstyle{remark}

\newcommand{\eps}{\varepsilon}

	\headers{High-Order AP Scheme for Kinetic Equations}{G. Dimarco, A. Klar, T. Köfler, L. Pareschi, and T. Sudarshan}

       \title {High-Order Asymptotic-Preserving Schemes for Kinetic Equations from Rarefied to Incompressible Regimes \thanks{The work of Theresa K\"{o}fler has been supported by the European Union's Horizon Europe research and innovation programme under grant agreement No. 101086214 (DATAHYKING project).
	The work of Axel Klar and Tiwari Sudarshan has been supported by German Research Foundation (DFG) through grant  KL-1105-34,  ``Asymptotic preserving higher-order  meshfree schemes for kinetic equations".
	The work of Lorenzo Pareschi was partially supported by the Royal Society under the Wolfson Fellowship ``Uncertainty quantification, data-driven simulations and learning of multiscale complex systems governed by PDEs" and by the FIS2023-01334 Advanced Grant ``Tackling complexity: advanced numerical approaches for multiscale systems with uncertainties" (ADAMUS). The work of Giacomo Dimarco has been supported by by the Italian Ministry of University and Research (MUR) through the PRIN 2022 project (No. 2022KKJP4X) ``Advanced numerical methods for time dependent parametric partial differential equations with applications". This work has been written within the activities of GNFM and GNCS groups of INdAM (Italian National Institute of High Mathematics). Program codes can be found on \url{https://github.com/theresakoefler/AP_schemes_kinetic_equs.git}.}}
	
\author{
  Giacomo Dimarco\thanks{Department of Mathematics and Computer Science, and Center for Modeling, Computing and Statistics, University of Ferrara, Italy (\email{giacomo.dimarco@unife.it}).}
  \and
  Axel Klar\thanks{Department of Mathematics, RPTU Kaiserslautern-Landau, Kaiserslautern, Germany (\email{klar@rptu.de}, \email{tiwari@rptu.de}).}
  \and
  Theresa K\"{o}fler\thanks{Department of Mathematics and Computer Science, University of Ferrara, Italy, and Department of Mathematics, RPTU Kaiserslautern-Landau, Germany (\email{theresa.kofler@unife.it}).}
  \and
  Lorenzo Pareschi\thanks{Department of Mathematics and Computer Science, University of Ferrara, Italy, and Maxwell Institute for Mathematical Sciences and Department of Mathematics, Heriot-Watt University, Edinburgh, UK (\email{l.pareschi@hw.ac.uk}).}
  \and
  Tiwari Sudarshan\footnotemark[3] 
}

\begin{document}
\maketitle

\begin{abstract} 
This work introduces a novel high-order numerical framework for solving kinetic equations, designed to remain uniformly valid across all regimes of the mean free path, spanning from the rarefied kinetic scale to the incompressible hydrodynamic limit. The method is built upon a micro-macro decomposition, which reformulates the underlying kinetic equation into a coupled system consisting of a macroscopic part, representing the fluid-dynamic evolution, and a microscopic part, describing the non-equilibrium deviations.
The proposed framework ensures high-order temporal accuracy through the use of Implicit-Explicit Runge–Kutta methods, which provide stability and efficiency in stiff regimes, while spatial resolution is enhanced by combining finite-difference WENO reconstructions with high-order central difference approximations. A key feature of the proposed methodology is its Asymptotic-Preserving (AP) property. We demonstrate that, in the appropriate asymptotic limit as the mean free path tends to zero, the scheme consistently reduces to a high-order finite-difference formulation of the incompressible Navier–Stokes equations.
To support the theoretical findings, a set of numerical experiments are performed on one- and two-dimensional benchmark problems, which confirm the accuracy, stability, and versatility of the method across different flow regimes.
\end{abstract}
\section{Introduction}
It is well known that microscopic particle dynamics can be linked to macroscopic quantities in fluid dynamics with the help of kinetic equations \cite{Chap,Cer,Levermore91}.  Macroscopic limits of kinetic equations are obtained when  the Knudsen number, which is representing the ratio of the mean free path of particles between collisions to some characteristic length of the flow, is small. The transition from kinetic  equations to incompressible fluid flow  with a  finite Reynolds number requires additionally a small Mach number.
In this case, the scaling of the kinetic equation must be handled with care in order to recover the correct fluid behaviour
\cite{Levermore91,SR}.

In the incompressible  fluid dynamic regime, fast acoustic waves, propagating at the speed of sound, become decoupled from the slow convective modes. From a computational perspective, this creates a "stiffness" issue, additionally to the small mean free path problem. This  means that standard numerical schemes treating only the smallness of the mean free path  in an implicit way,  require still prohibitively small time steps 
in the low Mach limit.
Furthermore, traditional stochastic methods, like Direct Simulation Monte Carlo (DSMC), suffer from 
 large statistical fluctuations, if the  macroscopic velocity becomes negligibly small compared to thermal velocities. This necessitates the development of deterministic methods and sophisticated decomposition strategies \cite{TATSIOS2025113500,Dimarco_acta}.

To address the challenges of the transition between kinetic and fluid descriptions, the micro--macro decomposition has emerged as a cornerstone technique. The approach has been developed for hyperbolic and diffusive limits of kinetic  equations and hyperbolic relaxation systems. The  approach splits the distribution function into two orthogonal components, namely the macroscopic component, typically a Maxwellian distribution, that is uniquely determined by the macroscopic quantities, and the microscopic part, that captures the rarefaction effects.
We refer among many other contributions to \cite{JinLevermore1996,JinPareschiToscani1998,KlarAP}
and 
\cite{BennouneLemou,CrouseillesLemou,LemouMieussens,K992,JinPareschiToscani1999,NaldiPareschiToscani2002,Mapundi}. 
We note that for kinetic equations  it is often  implemented with the Discrete Velocity Method (DVM), meaning the velocity domain is replaced by a finite set of discrete velocities \cite{Dimarco_acta,Illner,Mieussens}.

Micro-macro decompositions are used to obtain  asymptotic-preserving (AP) numerical methods, whose main objective is to ensure stability and accuracy uniformly across different asymptotic regimes, without the need to refine the discretization parameters according to the scale of the problem. In particular, AP schemes are designed so that, as the scaling parameter tends to zero, the numerical method automatically degenerates into a consistent discretization of the limiting macroscopic equations. 
In this context, Implicit–Explicit (IMEX) time integrator methods have emerged as a powerful tool for the time discretization of partial differential equations with stiff relaxation terms \cite{DimarcoAP1,DimarcoAP2,IMEX_PR,Bosc,Bosc1,Carp,AscherRuuthSpiteri1997,BoscarinoPareschiRusso2024,Ruuth,Multis,Bosc2}. These schemes typically treat stiff collision operators implicitly while handling transport terms explicitly, thus combining stability with computational efficiency. Unlike fully implicit methods, which are often too costly in high dimensions, IMEX schemes avoid the need to solve large nonlinear systems, while still permitting time steps that are not restricted by the small relaxation parameter. Moreover, when properly designed, IMEX methods inherit the AP property, guaranteeing a smooth transition from the kinetic to the hydrodynamic regime without loss of accuracy. Their flexibility and high-order accuracy make them particularly well suited for coupling with advanced spatial discretizations, such as weighted essentially non--oscillatory (WENO) \cite{Shu} or central difference schemes, in order to achieve robust and accurate solvers across a wide range of flow regimes.

There have been many developments of AP schemes for kinetic equations with diffusive scaling as mentioned above. AP schemes for the low-Mach limit of compressible Navier-Stokes equations have also been discussed by many authors. We cite only two examples \cite{DegondTang2011,DimarcoLoubereVignal2017}.
On the other hand, the full incompressible Navier Stokes limit for kinetic equations has 
not been discussed frequently  from a numerical perspective, see \cite{K992} for a low order AP method.

The objective of this work is to introduce a novel high-order numerical framework for solving kinetic equations, designed to remain uniformly valid across all regimes of the mean free path, spanning from the rarefied kinetic scale to the incompressible  hydrodynamic limit. In this low Mach limit, the method should reduce to an IMEX high-order finite difference solver for the incompressible Navier-Stokes equations. The method is built upon a micro-macro decomposition, which reformulates the underlying kinetic equation into a coupled system consisting of a macroscopic part, representing the fluid-dynamic evolution, and a microscopic part, describing the non-equilibrium deviations.
It combines nonoscillatory spatial discretizations with IMEX Runge-Kutta time integrators and high order numerical integrators for the DVM. In this setting, the nonoscillatory upwind discretization of the convective part of the relaxation system naturally reduces, in the asymptotic limit, to a high-order treatment of the nonlinear convective terms of the incompressible equations. More specifically, we present a family of numerical schemes based on a micro-macro decomposition with diffusive scaling, capable of achieving uniform accuracy in the low Mach number limit without refining the discretization in space and time. In the asymptotic limit, the proposed methods become high-order projection methods for the INS equations, solved on a collocated grid.

The rest of the paper is organized as follows. Section \ref{sec2} recalls the connection between kinetic theory and the INS equations while Section \ref{sec2.1} introduces the coupled system, obtained by the micro--macro decomposition, which constitutes a one to one map between a full kinetic description and a consistent fluid description. Section \ref{APfirst} introduces the asymptotic preserving method and in particular the first order time discretization. Sections \ref{sec3} and \ref{sec3b} are dedicated to higher order  time and space discretizations that yield the proposed family of numerical schemes, based on IMEX Runge-Kutta methods, which in turn, we show to lead to a projection formulation for the INS equations. High-order spatial discretizations and high order numerical integrators are  applied consistently within this framework. Section \ref{sec4} presents a set of numerical experiments that illustrate the accuracy, efficiency, and stability of the proposed approach. Finally, Section \ref{sec5} draws some conclusions and outlines directions for future research.
\section{Kinetic equations and low Mach number limit}\label{sec2}

We introduce here the underlying kinetic model and recall some results concerning its incompressible limit. These results will serve as the starting point for the derivation of the model in the next section and will subsequently be used to establish the numerical connection between the two models. For a theoretical analysis of kinetic equations and their asymptotic limits, we refer to \cite{Levermore91,SR,GSR}, while for related numerical discussions we refer to \cite{Klar99,JunkKlar,JinPareschiToscani1998}. 

\subsection{Kinetic equation and macroscopic limit}
We consider the following equation
\begin{equation}
	\partial_t F + v \cdot \nabla_x F = C(F), \label{eq:kinetic}
\end{equation}
where $F = F(x,v,t)$ is the so-called distribution function, which gives the probability for a particle to be located at the spatial position $x \in D \subset \mathbb{R}^d, d=1,2,3$, with velocity $v=(v_1,\dots,v_d) \in \mathbb{R}^d$, at time $t \in [0,\infty)$.
The operator $C(F)$ denotes a given collision operator, describing the interactions between particles. 
For this paper, 
we focus on the specific case of the BGK collision operator \cite{BGK}, which reads
\begin{equation}
	C(F)(v) = -\tfrac{1}{\tau} \Big( F(v) - M_F(v) \Big), \label{eq:bgk}
\end{equation}
where 
\[
M_F = \frac{\rho_F}{(2\pi T_F)^{\frac{d}{2}}} \exp \left(- \frac{|v-u_F|^2}{2T_F}\right)
\]
is the Maxwellian having the same moments with respect to $1, v, |v|^2$ as $F$, i.e.
with $$
\rho_F = \int F(x,v ) dv , \;  \rho_F u_F = \int v F(x,v ) dv , \; \rho_F u_F^2 + d \rho_F T_F = \int \vert v \vert^2 F(x,v ) dv 
$$
and  $\tau$ is the relaxation time.
 
Introducing the diffusive space–time scaling $x \mapsto x/\varepsilon$ and $t \mapsto t/\varepsilon^2$, where $\varepsilon$ is the Knudsen number (related to the mean free path between two consecutive collisions), one obtains the scaled equation
\begin{equation}
	\partial_t F + \frac{1}{\varepsilon} v \cdot \nabla_x F
	= \frac{1}{\varepsilon^2} C(F). \label{eq:scaled}
\end{equation}

We also introduce the Mach number $Ma$, denoting the ratio of the bulk velocity to the sound speed, and the finite Reynolds number $Re$, a dimensionless parameter representing the ratio of inertial to viscous forces within a fluid.
These numbers are related to $\varepsilon$ by the following relation:
\begin{equation}
	\varepsilon = \frac{Ma}{Re}. \label{eq:eps-ma-re}
\end{equation}
Thanks to a perturbation procedure (see, for example, \cite{Levermore91,Esposito}),
the limit equation for \eqref{eq:scaled} as $\varepsilon \to 0$ is obtained by searching for solutions of the form
\begin{equation}\label{eq:ansatz}
	F = M \big(1 + \varepsilon f \big),
\end{equation}
where $M$ is the normalized Maxwellian with zero drift,
\[
M(v) = \frac{1}{(2\pi)^{d/2}} \exp\!\left(-\frac{|v|^2}{2}\right).
\]
Substituting this ansatz into \eqref{eq:scaled} yields
\begin{equation}
	 \partial_t f + \frac{1}{\varepsilon} v \cdot \nabla_x f
	= \frac{1}{\varepsilon^3}\frac{1}{M}C\left(M \big(1 + \varepsilon f \big)\right).
\end{equation}
Note that 
\begin{align*}
	\rho_F = 1+ \epsilon <f>, \; \rho_F u_F = \epsilon < vf>, \; \rho_F \vert u_F\vert^2 + d \rho_F T_F= d+ \epsilon < \vert v \vert^2 f>
\end{align*}
with
\[
\langle f \rangle = \langle f \rangle_{L^2(M)} = \int_{\mathbb{R}^d} f(v) M(v) \, dv.
\]
Therefore, we have the relations
\begin{align}
	\label{relation}
	\rho_F = 1+ \epsilon \rho, \; \rho_F u_F = \epsilon u, \;  T_F= 1+ \epsilon \frac{T}{1+\epsilon \rho} - \epsilon^2 \frac{\vert u\vert^2}{d(1+\epsilon \rho)^2}=1+ \epsilon T + \mathcal{O}(\epsilon^2)
\end{align}
with
$  \rho = <f> $, $ u = <v f>$ , $  T = <\frac{\vert v\vert^2-d}{d}f> $. 
Then,  using the  projection
\begin{align}
	\label{projection}Pf=  \rho +  v \cdot u+ \frac{\vert v\vert^2-d}{2} T 
\end{align}
a Taylor expansion of $C$ gives 
\begin{align}
	\label{expansion}
	\frac{1}{M} C(M(1+\epsilon f)) = \epsilon L f + \epsilon^2 Q (f) 
\end{align}
with
$$
Lf = L(I-P)f= \frac{1}{\tau}(P-I) f \; 
$$
and the remainder term
\[
Q(f)(v) =Q(\rho,u,T)(v)= \tfrac{1}{\varepsilon^2 \tau} \left( M^{-1}M_F - 1 - \varepsilon Pf\right).
\] 
Thus, one obtains the following equation for $f$:
\begin{align}
\partial_t f + \frac{1}{\varepsilon} v \cdot \nabla_x f = \frac{1}{\varepsilon^2} L(f)(v) + \frac{1}{\varepsilon} Q(f)(v). \label{eq:expansion}
\end{align}

As $\epsilon \rightarrow 0$ one obtains  \cite{SR,Levermore91}, that the distribution function $f$ converges to a distribution $f_0$ of the form 
\[
f_0 = \rho_0 + v \cdot u_0 + \frac{|v|^2 - d}{2} T_0
\]
and that the limit equations for the macroscopic moments $\rho_0,u_0,T_0$ are the
incompressible Navier Stokes equations with Boussinesq approximation for the density and temperature, that is
\begin{align}
	\label{limit0}
	\nabla_x \cdot  u_0 =0 \\
	\nabla_x (\rho_0+T_0) =0 \nonumber\\
	\partial_t u _0+ \nabla_x \cdot u _0 \otimes u_0 + \nabla_x \tilde{p}= \tau  \Delta_x  u_0\nonumber\\
	\partial_t T _0+\nabla_x \cdot (u_0T_0) =  \tau \Delta_x T_0,\nonumber
\end{align}
where $\varepsilon\tilde{p} = \rho +T$.
In passing we note that a rigorous proof of convergence for the full Boltzmann collision operator has been obtained in \cite{GSR}.

\subsection{Micro-macro decomposition and fluid dynamic limit}\label{sec2.1}
We will now derive a system of equations under a diffusive space-time scaling based on the BGK model. As will be shown later, this approach leads to the incompressible Navier--Stokes equations in the low Mach number limit. 
To that aim, we use a decomposition of $f$ into a macroscopic and a microscopic part 
\[
f = Pf + (I-P)f,
\]
where $P$ denotes again the projection onto the space of collision invariants and that remains valid at any time, see \cite{LemouMieussens}. By using basic properties of the collision operator, we obtain the following coupled system
\begin{align}
\partial_t Pf + \frac{1}{\varepsilon} \nabla_x \cdot P(vf) &= 0 \label{equ:Pf}\\
\partial_t (I-P)f + \frac{1}{\varepsilon} \nabla_x \cdot (I-P)(vf) &= - \frac{1}{\varepsilon^2 \tau} (I-P)f + \frac{1}{\varepsilon} Q(f),\label{equ:I-Pf}
\end{align}
with 
\[
(I-P)(vf) = (I-P)(vP(f)) + (I-P)(v(I-P)(f))
\]
and 
\[
(I-P)(vP(f)) = \mathcal{A}(v) \langle vf \rangle + \frac{\mathcal{B}(v)}{2} \langle \frac{v^2 - d}{d} f\rangle.
\]
with the polynomials 
$$
\mathcal{A}(v) = v \otimes v -\frac{\vert v \vert^2 }{d}I
$$
and 
$$
\mathcal{B}(v) = v (\vert v \vert^2 -(d+2)).
$$
Moreover, we obtain
\[
(I-P)(v(I-P)(f)) = v(I-P)f - \langle \mathcal{A}(v) (I-P)f \rangle v - \frac{|v|^2 - d}{2d} \langle\mathcal{B}(v)(I-P)f\rangle.
\]
Using moments of $f$ in order to rewrite \eqref{equ:Pf}, we get
\begin{align}
\varepsilon \partial_t \langle f \rangle +  \nabla_x \cdot \langle v f \rangle &= 0 \nonumber \\
\varepsilon \partial_t \langle v f \rangle + \nabla_x \cdot \langle v \otimes v f \rangle +  \nabla_x (T+\rho) &= 0  \\
\varepsilon \partial_t \langle (|v|^2 - (d+2))f \rangle +  \nabla_x \cdot \langle v(|v|^2 - (d+2))f \rangle &= 0, \nonumber
\end{align}
or
\begin{align}
\varepsilon \partial_t \rho +  \nabla_x \cdot u &= 0 \nonumber \\
\varepsilon \partial_t u + \nabla_x \cdot \langle \mathcal{A}(v) (I-P)f \rangle +  \nabla_x (T+\rho) &= 0 \label{eq:moments} \\
\varepsilon \partial_t S +  \nabla_x \cdot \langle \mathcal{B}(v) (I-P)f \rangle &= 0, \nonumber
\end{align}
where we use the notation $S = \langle (|v|^2 - (d+2))f \rangle = dT - 2\rho$.
Rewriting \eqref{equ:I-Pf} gives
\begin{align*}
\varepsilon^2 \partial_t (I-P)f + \varepsilon \nabla_x \cdot \left[ \mathcal{A}(v) u + \frac{\mathcal{B}(v)}{2} T + (I-P) (v(I-P)f)\right] = - \frac{1}{\tau} (I-P)f + \varepsilon Q(f).
\end{align*}
Defining 
\[
g = \frac{1}{\varepsilon} (I-P)f
\]
we obtain
\begin{align}
\varepsilon^2 \partial_t g +  \nabla_x \cdot \left[ \mathcal{A}(v) u + \frac{\mathcal{B}(v)}{2} T + \varepsilon (I-P) (vg)\right] = - \frac{1}{\tau} g + Q(f). \label{eq:g}
\end{align}
Combining \eqref{eq:moments} and \eqref{eq:g} gives
\begin{align}
\varepsilon \partial_t \rho +  \nabla_x \cdot u &= 0 \nonumber \\
\varepsilon \partial_t u + \varepsilon \nabla_x \cdot \langle \mathcal{A}(v) g \rangle +  \nabla_x (T+\rho) &= 0\nonumber \label{eq:moments2} \\
\partial_t T + \frac{2}{\varepsilon d} \nabla_x \cdot u + \frac{1}{d}\nabla_x \cdot \langle \mathcal{B}(v) g \rangle &= 0  \\  \nonumber
\varepsilon^2 \partial_t g +  \nabla_x \cdot \left[ \mathcal{A}(v) u + \frac{\mathcal{B}(v)}{2} T + \varepsilon (I-P) (vg)\right] &= - \frac{1}{\tau} g + Q(f).
\end{align}
\begin{theorem}
In the limit $\varepsilon \rightarrow 0$, the system of partial differential equations \eqref{eq:moments2} leads to the incompressible heat conducting Navier-Stokes equations.
\end{theorem}
\begin{proof}
In the limit $\varepsilon \rightarrow 0$, \eqref{eq:moments2} gives with $p = \epsilon \tilde p = \rho +T $
\begin{align*}
  \nabla_x \cdot u &= 0 \nonumber \\
  \nabla_x(\rho +T)&= 0 \nonumber \\
  \partial_t u + \nabla_x \cdot \langle \mathcal{A}(v) g \rangle + \nabla_x \tilde p &= 0\nonumber \label{eq:moments3} \\
\partial_t S  +\nabla_x \cdot \langle \mathcal{B}(v) g \rangle &= 0
\end{align*}
and
\begin{align*}
g&= - \tau  \nabla_x \cdot \left[ \mathcal{A}(v) u + \frac{\mathcal{B}(v)}{2} T \right]  + \tau Q(f).
\end{align*}

Note that  $S=dT-2 \rho = (d+2) T -2 p =(d+2) T -2 \epsilon \tilde p $
and note the following properties of the polynomials $\A$ and $\B$
 $$
< \B(v) \otimes \B(v) > = 2 (d+2) I 
$$
and
$$
<\A(v) \otimes \A(v) >  : \nabla_x u  = \nabla_x u +(\nabla_x u )^T  -  \frac{2}{d} \nabla_x \cdot  u I.
$$
Moreover, as $\epsilon \rightarrow 0$ one has 
$$
<\mathcal{A}(v) Q (f)>\  = \frac{1}{\epsilon^2 \tau }<\mathcal{A}(v) M_F> \rightarrow   \frac{1}{\tau} \mathcal{A}(u)
$$
and 
$$
<\mathcal{B}(v) Q (f)>  =\frac{1}{\epsilon^2 \tau }<\mathcal{B}(v) M_F>    \rightarrow  \frac{1}{\tau} (d+2) (u T).
$$
Altogether this gives in the limit $\epsilon \rightarrow 0$
\[
\langle \mathcal{A}(v) g\rangle  = - \tau (\nabla_x u + (\nabla_x u)^T) + \mathcal{A} (u) 
\]
and
\[
\langle \mathcal{B}(v) g\rangle  = - \tau (d+2) \nabla_x T + (d+2)(uT)
\]
and finally the thermal Navier-Stokes equations (\ref{limit0}).
\end{proof}

\section{The asymptotic preserving numerical  method}
\label{APfirst}
We start by treating  only the stiff low-Mach limit and the stiffness in the Knudsen number implicitly. The scheme for the limit equations is a projection scheme treating the viscous terms explicitly. A version yielding an implicit treatment of the viscous terms in the limit is discussed afterwards.
We separate  the implicit and explicit components as follows
\begin{equation}
\begin{split} \label{implicit-explicit}
\partial_t \rho +  \underbrace{\frac{1}{\varepsilon}\nabla_x \cdot u}_{implicit} &= 0 \\
\partial_t u + \underbrace{ \nabla_x \cdot \langle \mathcal{A}(v) g \rangle + \frac{1}{\varepsilon} \nabla_x (T+\rho)}_{implicit} &= 0 \\
\partial_t T +  \underbrace{\frac{2}{\varepsilon d} \nabla_x \cdot u + \frac{1}{d} \nabla_x \cdot \langle \mathcal{B}(v) g \rangle}_{implicit} &= 0\\
\varepsilon^2 \partial_t g +  \underbrace{\nabla_x \cdot \left[ \mathcal{A}(v) u + \frac{\mathcal{B}(v)}{2}T + \varepsilon (I-P)(vg) \right]}_{explicit} &= -  \underbrace{\frac{1}{\tau} g}_{implicit} + \underbrace{Q(\rho,u,T)}_{explicit}
\end{split}
\end{equation}
Thus, for the pressure we have 
\begin{equation}
\begin{split} 
\partial_t p +  \underbrace{\frac{d+2}{\varepsilon d} \nabla_x \cdot u + \frac{1}{d} \nabla_x \cdot \langle \mathcal{B}(v) g \rangle}_{implicit} &= 0\\
\end{split}
\end{equation}

\subsection{First-order time discretization and AP property}
The time discretization is then
\begin{align}
\frac{1}{\Delta t } (u^{n+1/2}- u^{n})+   \nabla_x \cdot \langle \mathcal{A}(v) g^{n+1}  \rangle =0 \nonumber\\
\frac{1}{\Delta t } (u^{n+1}- u^{n+1/2})+ \frac{1}{\epsilon} \nabla_x p^{n+1}=0\label{us}\\
\frac{1}{\Delta t } (T^{n+1/2}- T^{n})+ \frac{1}{d}\nabla_x\cdot \langle \mathcal{B}(v) g^{n+1} \rangle= 0\nonumber\\
\frac{1}{\Delta t } (T^{n+1}- T^{n+1/2})+ \frac{2}{\epsilon d}\nabla_x\cdot u^{n+1} = 0\nonumber\\
\frac{1}{\Delta t } (\rho^{n+1}- \rho^{n})+ \frac{1}{\epsilon }\nabla_x\cdot u^{n+1} = 0\nonumber
\end{align}
This leads for the pressure to
\begin{align}
\frac{1}{\Delta t } (p^{n+1/2}- p^{n})+ \frac{1}{d}\nabla_x\cdot \langle \mathcal{B}(v) g^{n+1} \rangle = 0\nonumber\\
\frac{1}{\Delta t } (p^{n+1}- p^{n+1/2})+ \frac{(d+2)}{\epsilon d}\nabla_x\cdot u^{n+1} = 0\label{ps}
\end{align}
Moreover,
\begin{eqnarray*}
	\frac{\varepsilon^2}{\Delta t} (g^{n+1} -g^n) +  \nabla_x \cdot \left[ \mathcal{A}(v) u^n + \frac{\mathcal{B}(v)}{2}T^n + \varepsilon (I-P)(vg^n) \right]&= -  \frac{1}{\tau} g^{n+1} + Q(\rho^n,u^n,T^n)
\end{eqnarray*}
which gives 
with $\gamma = \frac{1}{\epsilon^2 + \frac{\Delta t}{\tau} }=  \frac{\tau }{\epsilon^2 \tau + \Delta t}$
and $Q^n =Q(\rho^n,u^n,T^n)$
\begin{eqnarray}
\label{Gamma}
g^{n+1}= \epsilon^2 \gamma g^n+ \Delta t \gamma Q^n-\Delta t \gamma \nabla_x \cdot \left[ \mathcal{A}(v) u^n + \frac{\mathcal{B}(v)}{2}T^n + \varepsilon (I-P)(vg^n) \right].
\end{eqnarray}
The discrete equations can be rearranged in a more efficient way:
Compute first $u^{n+1/2},T^{n+1/2}$ using 
\begin{align*}
	\frac{1}{\Delta t } (u^{n+1/2}- u^{n})+   \nabla_x \cdot \langle \mathcal{A}(v) g^{n+1}  \rangle  =0 \nonumber\\
	\frac{1}{\Delta t } (T^{n+1/2}- T^{n})+ \frac{1}{d}\nabla_x\cdot \langle \mathcal{B}(v) g^{n+1}  \rangle   = 0\nonumber
\end{align*}
or 
%
%
\begin{align}
\label{e1}
\frac{1}{\Delta t} (u^{n+1/2} - u^{n}) &+ \epsilon^2 \gamma \nabla_x \cdot \langle \mathcal{A}(v) g^{n} \rangle - \Delta t \gamma \left[ \Delta_x u^n + \frac{d-2}{d} \nabla_x (\nabla_x \cdot u^n) \right] \nonumber \\
& + \Delta t \gamma \nabla_x \cdot \langle \mathcal{A}(v) Q^{n} \rangle = 0, \\[1.5ex]
\label{e2}
\frac{1}{\Delta t} (T^{n+1/2} - T^{n}) &+ \frac{\epsilon^2 \gamma}{d} \nabla_x \cdot \langle \mathcal{B}(v) g^{n} \rangle - \Delta t \gamma \frac{d+2}{d} \Delta_x T^{n} \nonumber \\
& + \frac{\Delta t \gamma}{d} \nabla_x \cdot \langle \mathcal{B}(v) Q^{n} \rangle = 0.
\end{align}
Then, we obtain $p^{n+1/2}$ simply via 
\begin{eqnarray}
\label{pn12}
p^{n+1/2}- p^{n}=T^{n+1/2}- T^{n}.
\end{eqnarray}
Moreover,
\begin{eqnarray}
\label{rhop}
\frac{d+2}{d} (\rho^{n+1}- \rho^{n})=p^{n+1}- p^{n+1/2}.
\end{eqnarray}
Combining (\ref{ps}) and (\ref{us}) and the above gives
\begin{eqnarray}
\label{ppoisson}
\frac{\epsilon^2}{\Delta t^2} p^{n+1}- \frac{d+2}{d}\Delta_x  p^{n+1}  = \frac{\epsilon^2}{\Delta t^2} p^{n+1/2}   - \frac{(d+2) \epsilon }{d \Delta t } \nabla_x \cdot u^{n+1/2}.
\end{eqnarray}
Then, (\ref{rhop}) can again be used to compute $\rho^{n+1}$ and finally
$T^{n+1}$ via
\begin{eqnarray}
\label{Tp}
T^{n+1}- T^{n+1/2}=\frac{2}{d} (\rho^{n+1}- \rho^{n}).
\end{eqnarray}
The numerical algorithm is 
\begin{algorithm}\caption{AP algorithm, first-order time discretization}
	Start with $\rho^n,u^n,T^n$ and $g^{n}$ 
	\begin{enumerate}
		\item Compute $g^{n+1}$  via (\ref{Gamma}).
		\item Compute
		$u^{n+1/2}, T^{n+1/2}$ explicitly via 	(\ref{e1}) and (\ref{e2}).
		Compute $p^{n+1/2}$ via (\ref{pn12}).
		\item Compute  $\tilde  p^{n+1}=\frac{p^{n+1}}{\epsilon}$ using (\ref{ppoisson})  via 
		\begin{eqnarray}
		\frac{\epsilon^2}{\Delta t^2} \tilde p^{n+1}- \frac{d+2}{d}\Delta_x  \tilde p^{n+1}  = \frac{\epsilon}{\Delta t^2} p^{n+1/2}   - \frac{(d+2) }{d \Delta t } \nabla_x \cdot u^{n+1/2}
		\end{eqnarray}
		and $u^{n+1} $ as 
		$$u^{n+1}= u^{n+1/2}- \Delta t \nabla_x  \tilde p^{n+1}.$$
		\item 
		Compute  $\rho^{n+1}$ via 
		\begin{eqnarray}
		\frac{d+2}{d} (\rho^{n+1}- \rho^{n})=\epsilon \tilde p^{n+1}- p^{n+1/2}
		\end{eqnarray}
		or directly via
		\begin{eqnarray}
		\frac{1}{\Delta t } (\rho^{n+1}- \rho^{n})+\frac{1}{\epsilon} \nabla_x \cdot u^{n+1}=0
		\end{eqnarray}
		and 
		$ T^{n+1}$ via  (\ref{Tp}).
	\end{enumerate}	
\end{algorithm}
\begin{remark}
	As $\epsilon \rightarrow 0$ this gives with  $\gamma \rightarrow \frac{\tau}{\Delta t }$ the following algorithm.
	Compute 		
	\begin{eqnarray*}
		g^{n+1}=  \tau Q(\rho^n,u^n,T^n)-\tau \nabla_x \cdot \left[ \mathcal{A}(v) u^n + \frac{\mathcal{B}(v)}{2}T^n \right]
	\end{eqnarray*}
	and
	$T^{n+1/2}$ via 	
	\begin{eqnarray*}
		T^{n+1/2} 
		= T^{n} -  \Delta t \frac{d+2}{d}[\nabla_x \cdot (u^nT^n)-  \tau \Delta_x T^n]
	\end{eqnarray*}
	and
	\begin{eqnarray}
	u^{n+1/2} =  u^n - \Delta t [\nabla_x \cdot \A(u^n ) - \tau \Delta_x u^n]
	\end{eqnarray}
	Compute	  \begin{eqnarray*}
		\Delta_x \tilde p^{n+1}  =  \frac{1 }{ \Delta t } \nabla_x \cdot u^{n+1/2}
	\end{eqnarray*} 
	and then
	$$
	u^{n+1} = u^{n+1/2}- \Delta t \nabla_x \tilde p,
	$$	
	which gives 
	$\nabla_x \cdot u^{n+1} =0$.
	Finally,  for $\epsilon \rightarrow 0$ we have 
	\begin{eqnarray*}
		T^{n+1}-T^n =	T^{n+1}-T^{n+1/2}+T^{n+1/2}-T^n \\= (-\frac{2}{d+2}+1)(T^{n+1/2}-T^n )
		=-  \Delta t[\nabla_x \cdot (u^nT^n)-  \tau \Delta_x T^n]
	\end{eqnarray*}
	due to
	\begin{eqnarray*}
		T^{n+1}-T^{n+1/2}=\frac{2}{d} (\rho^{n+1}- \rho^{n})
		=\frac{2}{d+2}(\epsilon \tilde p^{n+1}- p^{n+1/2})\\
		=-\frac{2}{d+2}( p^{n+1/2}- \epsilon \tilde p^{n})
		= -\frac{2}{d+2}(T^{n+1/2}-T^n )
	\end{eqnarray*}
	in the limit $\epsilon \rightarrow 0$ .
	Thus, 	for $\epsilon \rightarrow 0$
	\begin{eqnarray*}
		T^{n+1}= T^n-  \Delta t[\nabla_x \cdot (u^nT^n)-  \tau \Delta_x T^n],
	\end{eqnarray*}
	We have  obtained 	a standard projection scheme for the incompressible Navier-Stokes equations with an explicit treatment of the diffusive terms in the thermal Navier--Stokes limit.

\end{remark}

\begin{remark}
	\label{remCFL}
	Following the analysis in \cite{Bosc1} one observes that for stability the above scheme requires for large
	$\epsilon$
	a time step restriction like  the original kinetic equation, that means we choose 
	\begin{equation}
		\label{CFL}
		\Delta t  =  CFL \frac{ \Delta x \epsilon }{v_{max} }.
	\end{equation}
	For smaller values of $\epsilon $ and, in particular for $\epsilon $ near $0$ the scheme requires the same
	parabolic time step restriction as the explicitly discretized  limit equations, that means a time step $\Delta t$ chosen as 
	\begin{equation}
		\label{CFL2}
		\Delta t  =  CFL  \min\{\frac{ \Delta x}{\max \vert u \vert} , \frac{( \Delta x)^2}{2 \tau}\}.
	\end{equation}
\end{remark}

\begin{remark}[Additional implicit discretization treating the parabolic stiffness]
\label{pstiff}
Solving two additional Helmholtz problems instead of (\ref{e1}) and (\ref{e2})  a projection scheme for the thermal Navier-Stokes system with an implicit treatment of the 
diffusive terms can be obtained avoiding the parabolic stiffness in (\ref{CFL2}).
\end{remark}

\section{Higher order numerical methods - Time discretization}\label{sec3}
The purpose of this section is to introduce higher order  numerical methods which provide uniformly stable approximations with respect to the small parameter $\varepsilon$. In the asymptotic limit 
$\varepsilon\to 0$, these methods yield high–order accurate approximations of the incompressible Navier–Stokes equations.

The main idea is to use a suitable combination of implicit – explicit (IMEX) Runge - Kutta discretizations \cite{IMEX_PR,Bosc1}. The main advantage of the proposed strategy is that, compared to a fully explicit method, the time step is not restricted by the stiffness of the system, while compared to a fully implicit scheme, the solution procedure is more straightforward, since the system to be solved is less involved. We will show that this choice is sufficient to obtain a projection scheme for the incompressible Navier–Stokes equations in the mean–free–path limit. 

\subsection{General IMEX schemes}
The general IMEX–RK schemes can be conveniently represented by a double Butcher tableau of the following form:
\begin{align}
	\text{Explicit: } 
	\begin{array}{c|c}
		\tilde{c} & \tilde{A} \\ \hline \\ [-1.9ex]
		& \tilde{w}^T
	\end{array}
	\qquad
	\text{Implicit: }
	\begin{array}{c|c}
		c & A \\ \hline \\[-1.9ex]
		& w^T
	\end{array}
\end{align}
For the explicit scheme, $\tilde{A}$ is an $s \times s$ lower–triangular matrix with zero diagonal entries.
For the implicit scheme, the $s \times s$ matrix $A$ is restricted to the class of diagonally implicit Runge–Kutta (DIRK) methods, which means that $a_{ij} = 0$ for $j > i$. The vectors $w$ and $\tilde{w}$ contain the quadrature weights used to combine the internal stages of the Runge–Kutta method, while the coefficients $c$ and $\tilde{c}$ are defined as
\[c_i = \sum_{j=1}^{i} a_{ij}, \quad \tilde{c}_i = \sum_{j=1}^{i-1} \tilde{a}_{ij}.\]
The following definitions characterize the IMEX schemes and the properties we aim to consider in this work. For further details we refer for instance to \cite{IMEX_PR}.
\begin{definition}
	An IMEX RK method is said to be of type A, if the matrix $A \in \mathbb{R}^{s \times s}$ is invertible, or equivalently $a_{ii} \neq 0$, $i = 1,\dots,s$.
\end{definition}
\begin{definition}
	An IMEX RK method is said to be of type CK, if the matrix $A \in \mathbb{R}^{s \times s}$ can be written as
	\begin{align*}A=\left(
		\begin{matrix}
			0 & 0 \\ a & \hat{A}
		\end{matrix}\right)
	\end{align*}
	with $a = (a_{21},\dots,a_{s1})^T \in \mathbb{R}^{(s-1)}$ and the submatrix $\hat{A} \in \mathbb{R}^{(s-1) \times (s-1)}$ invertible. If additionally $a = 0$, the scheme is said to be of type ARS.
\end{definition}
\begin{definition}
	An IMEX RK method is called implicitly stiffly accurate (ISA), if the corresponding DIRK method is stiffly accurate, which means that
	\begin{align*}
		a_{si} = w_i, \qquad i = 1,\dots,s.
	\end{align*}
	If additionally the explicit method satisfies
	\begin{align}
		\tilde{a}_{si} = \tilde{w}_i, \qquad i = 1,\dots,s-1,
	\end{align}
	the IMEX RK method is called globally stiffly accurate (GSA).
\end{definition}
Note that the numerical solution of a GSA IMEX RK scheme coincides exactly with the last internal stage \cite{IMEX_PR},\cite{Hairer}. We now apply the above general IMEX-RK approach to the system \eqref{eq:moments2}. Note that we add an additional equation for $p$, that is derived by summing the equations for $\rho$ and $T$, in the following for reasons of being able to implement IMEX RK without any operator splitting.

\subsection{Higher order AP time discretization}
We proceed as before by separating the implicit and the explicit components as in (\ref{implicit-explicit}).
So, we obtain now  for the internal stages: 
\begin{align}
\label{InternalStages}
\rho^{(i)} &= \rho^n - \frac{\Delta t}{\varepsilon} \sum_{j = 1}^i a_{ij} \nabla_x \cdot u^{(j)}, \nonumber \\
u^{(i)} &= u^n - \Delta t  \sum_{j = 1}^i \left( a_{ij} \nabla_x \cdot \langle \mathcal{A}(v) g^{(j)} \rangle  + \frac{a_{ij}}{\varepsilon} \nabla_x (T^{(j)} + \rho^{(j)})\right), \nonumber \\
T^{(i)} &= T^n - \Delta t  \sum_{j = 1}^i \left( \frac{a_{ij}}{d} \nabla_x \cdot \langle \mathcal{B}(v) g^{(j)} \rangle  + \frac{2 a_{ij}}{d \varepsilon} \nabla_x \cdot u^{(j)}\right), \nonumber \\
p^{(i)} &= p^n - \Delta t  \sum_{j = 1}^i \left( \frac{a_{ij}}{d} \nabla_x \cdot \langle \mathcal{B}(v) g^{(j)} \rangle  + \frac{d+2}{d \varepsilon} a_{ij} \nabla_x \cdot u^{(j)}\right), \\
g^{(i)} &= g^n - \frac{\Delta t}{\varepsilon^2} \sum_{j=1}^{i-1} \tilde{a}_{ij} \Bigl( \nabla_x \cdot \left[ \mathcal{A}(v)u^{(j)} + \tfrac{\mathcal{B}(v)}{2}T^{(j)} + \varepsilon (I-P) (vg^{(j)})\right] \nonumber \\
&\quad - Q(u^{(j)},\rho^{(j)},T^{(j)}) \Bigr) - \frac{\Delta t}{\varepsilon^2 \tau} \sum_{j=1}^i a_{ij} g^{(j)}. \nonumber
\end{align}

\begin{theorem}\label{ThmFullBGKTypeA}
	If the IMEX method is of type A and satisfies the GSA property, then in the fluid dynamic limit $\varepsilon \rightarrow 0$ the IMEX RK scheme \eqref{InternalStages} becomes a consistent discretization for the thermal  incompressible Navier Stokes equations.
\end{theorem}
\begin{proof}
Rewriting the last equation of \eqref{InternalStages} gives
\begin{align}
	\label{InternalStagesForg}
&g^{(i)} = \, \frac{\varepsilon^2 \tau}{\varepsilon^2 \tau + a_{ii} \Delta t} g^n \nonumber \\
&- \frac{\Delta t \tau}{\varepsilon^2 \tau + a_{ii} \Delta t} \sum_{j=1}^{i-1} \tilde{a}_{ij} \Biggl( \nabla_x \cdot \left[ \mathcal{A}(v) u^{(j)} + \frac{\mathcal{B}(v)}{2} T^{(j)} + \varepsilon (I-P)(vg^{(j)}) \right] - Q^{(j)} \Biggr) \\
&- \frac{\Delta t}{\varepsilon^2 \tau + a_{ii} \Delta t} \sum_{j=1}^{i-1} a_{ij} g^{(j)},\nonumber
\end{align}
where $Q^{(j)} = Q(u^{(j)},\rho^{(j)},T^{(j)})$ for readability. Then, by inserting  this 
into the second equation of \eqref{InternalStages}, using orthogonality of $\mathcal{A}(v)$ and $\mathcal{B}(v)$  and taking the limit $\varepsilon \to 0$ we obtain
\begin{equation}
\begin{split}
u^{(i)} = \, & u^n - \Delta t  \sum_{j = 1}^{i-1} a_{ij} \nabla_x \cdot \langle \mathcal{A}(v) g^{(j)} \rangle \\
& + \Delta t \tau  \sum_{j=1}^{i-1} \tilde{a}_{ij} \nabla_x \cdot \left( \langle \mathcal{A}(v) \otimes \mathcal{A}(v) \rangle : \nabla_x u^{(j)} - \langle \mathcal{A}(v) Q^{(j)} \rangle \right) \\
& + \Delta t  \sum_{j = 1}^{i-1} a_{ij} \nabla_x \cdot \langle \mathcal{A}(v) g^{(j)} \rangle - \frac{\Delta t}{\varepsilon} \sum_{j = 1}^i a_{ij} \nabla_x (T^{(j)} + \rho^{(j)}).
\end{split}
\end{equation}
Rescaling $p$ as  $\varepsilon \tilde p $ and, for simplicity, dropping  $\widetilde{~}$ ~in what follows, we obtain
\begin{equation}
\begin{split} \label{InternalStagesforu}
u^{(i)} &= u^n 
+ \Delta t \tau  \sum_{j=1}^{i-1} \tilde{a}_{ij}  \nabla_x \cdot \left( \langle \mathcal{A}(v) \otimes  \mathcal{A}(v) \rangle : \nabla_x u^{(j)} - \langle \mathcal{A}(v) Q^{(j)} \rangle\right) 
- \Delta t  \sum_{j = 1}^i  a_{ij} \nabla_x p^{(j)}.
\end{split}
\end{equation}
We now consider the stage equations for $p$, that can be directly obtained from the equations for $\rho$ and $T$.
\begin{equation}
\varepsilon p^{(i)} = \varepsilon p^n  - \Delta t  \sum_{j = 1}^i \left( \frac{1}{d} a_{ij} \nabla_x \cdot \langle \mathcal{B}(v) g^{(j)} \rangle  + \frac{2+d}{d \varepsilon} a_{ij} \nabla_x \cdot u^{(j)}\right) 
\end{equation}
Substituting $u^{(i)}$ gives
\begin{equation}
\label{eq:pressure_stage}
\begin{split}
p^{(i)} = \, & p^n - \frac{\Delta t}{d \varepsilon} \sum_{j = 1}^i a_{ij} \nabla_x \cdot \langle \mathcal{B}(v) g^{(j)} \rangle 
- \frac{(d+2)\Delta t}{d \varepsilon^2} \sum_{j = 1}^{i-1} a_{ij} \nabla_x \cdot u^{(j)} \\
& - \frac{d+2}{d} \frac{\Delta t}{\varepsilon^2} a_{ii} \nabla_x \cdot u^{(i)}_* 
+ \frac{d+2}{d} \frac{\Delta t^2}{\varepsilon^2} a_{ii}^2 \Delta_x p^{(i)}.
\end{split}
\end{equation}
where
\begin{equation}
u^{(i)}_* = u^n - \Delta t  \sum_{j = 1}^i \left( a_{ij} \nabla_x \cdot \langle \mathcal{A}(v) g^{(j)} \rangle \right) - \Delta t  \sum_{j = 1}^{i-1} \left( a_{ij} \nabla_x p^{(j)} \right).
\end{equation}
Rewriting gives
\begin{equation}
\label{InternalStagesforp}
\begin{split}
\frac{d+2}{d} a_{ii}^2 \Delta_x p^{(i)} - \frac{\varepsilon^2}{\Delta t^2 } p^{(i)} = \, & -\frac{\varepsilon^2}{\Delta t^2 } p^n + \frac{\varepsilon}{d \Delta t } \sum_{j = 1}^i a_{ij} \nabla_x \cdot \langle \mathcal{B}(v) g^{(j)} \rangle \\
& + \frac{d+2}{d \Delta t } \sum_{j = 1}^{i-1} a_{ij} \nabla_x \cdot u^{(j)} + \frac{(d+2)a_{ii}}{d \Delta t } \nabla_x \cdot u^{(i)}_*.
\end{split}
\end{equation}
Taking the divergence of the second equation of \eqref{InternalStages} gives for the first stage
\begin{equation} \label{SecondProofReference}
\nabla_x \cdot u^{(1)} = \nabla_x \cdot u^n - \Delta t  a_{11} \nabla_x \cdot \nabla_x \cdot \langle \mathcal{A}(v) g^{(1)} \rangle
- \Delta t a_{11} \Delta_x p^{(1)}.
\end{equation}
In the limit $\varepsilon \to 0$, we get for the first stage of \eqref{InternalStagesforp}
\begin{equation}
\Delta t a_{11} \Delta_x p^{(1)} =  \nabla_x \cdot u^n - \Delta t  a_{11} \nabla_x \cdot \nabla_x \cdot \langle \mathcal{A}(v) g^{(1)} \rangle
\end{equation}
and therefore $\nabla_x \cdot u^{(1)} = 0$.
For the inductive step $(i-1) \mapsto i$, the second equation of \eqref{InternalStages} gives
\begin{equation}
\begin{split}
\nabla_x \cdot u^{(i)} &= \nabla_x \cdot u^n 
- \Delta t  \sum_{j = 1}^i \left( a_{ij} \nabla_x \cdot \nabla_x \cdot \langle \mathcal{A}(v) g^{(j)} \rangle \right) - \Delta t  \sum_{j = 1}^{i-1} \left( a_{ij} \Delta_x p^{(j)} \right) 
- \Delta t  a_{ii} \Delta_x p^{(i)}.
\end{split}
\end{equation}
Using \eqref{InternalStagesforp} and the induction hypothesis $\nabla_x \cdot u^{(i-1)} = 0$ gives in the limit 
\begin{equation}
\begin{split}
\Delta t a_{ii} \Delta_x p^{(i)}  &= \nabla_x \cdot  u^n - \Delta t  \sum_{j = 1}^i \left( a_{ij} \nabla_x \cdot \nabla_x \cdot \langle \mathcal{A}(v) g^{(j)} \rangle \right) - \Delta t  \sum_{j = 1}^{i-1} \left( a_{ij} \Delta_x p^{(j)} \right).
\end{split}
\end{equation}
and therefore we obtain $\nabla_x \cdot u^{(i)} = 0$ in the incompressible limit.
By using the following properties
\begin{equation*}\begin{split}
\nabla_x \cdot \langle \mathcal{A}(v) Q \rangle &= \frac{1}{\tau}\left( \nabla_x \cdot \mathcal{A}(u)\right) = \frac{1}{\tau} \left((\nabla_x \cdot u) u + (u \cdot \nabla_x) u - \nabla_x \frac{|u|^2}{d} \right)\\
\nabla_x \cdot \langle \mathcal{A}(v) \otimes \mathcal{A}(v) \rangle : \nabla_x u &= \Delta_x u + \nabla_x(\nabla_x \cdot u) - \frac{2}{d} \nabla_x (\nabla_x \cdot u)
\end{split}
\end{equation*}
as well as the fact that all internal stages of $u$ are divergence-free in the incompressible limit, we obtain
\begin{equation}
\begin{split} \label{FinalForu_}
u^{(i)} &= u^n 
+ \Delta t  \sum_{j=1}^{i-1} \tilde{a}_{ij}  \left( \tau \Delta_x u^{(j)} -  (u^{(j)} \cdot \nabla_x ) u^{(j)}\right) 
+ \Delta t  \sum_{j=1}^{i-1} \tilde{a}_{ij} \nabla_x \frac{|u^{(j)}|^2}{d} 
- \Delta t  \sum_{j = 1}^i  a_{ij} \nabla_x p^{(j)}.
\end{split}
\end{equation}
Using Bernoulli's relation finally gives 
\begin{equation}
\begin{split} \label{FinalForu}
u^{(i)} &= u^n 
+ \Delta t  \sum_{j=1}^{i-1} \tilde{a}_{ij}  \left( \tau \Delta_x u^{(j)} -  (u^{(j)} \cdot \nabla_x ) u^{(j)}\right) 
- \Delta t  \sum_{j = 1}^i  a_{ij} \nabla_x p^{(j)},
\end{split}
\end{equation}
where $p$ now denotes the static pressure.
We now consider the stage equations for $p$. Taking the limit $\varepsilon \to 0$ and using the divergence-free condition for all internal stages of $u$ as well as Bernoulli's relation, we get
\begin{equation}\begin{split} \label{FinalForp}
  \sum_{j = 1}^{i} \left( a_{ij} \Delta_x p^{(j)} \right)  &=  \frac{1}{\Delta t}\nabla_x \cdot \left( u^n + \Delta t  \sum_{j=1}^{i-1} \tilde{a}_{ij}  
 \left( \tau \Delta_x u^{(j)} -  (u^{(j)} \cdot  \nabla_x ) u^{(j)}\right)  \right).
\end{split}
\end{equation} 
Finally, we consider the stage equations for $T$.
\begin{equation}
		T^{(i)} = T_*^{(i)}- \Delta t  \sum_{j = 1}^i \left(  \frac{2}{d \varepsilon} a_{ij} \nabla_x \cdot u^{(j)}\right) 
\end{equation}
with
\begin{equation}
	\label{Tstar}
	T_*^{(i)} = T^n - \Delta t  \sum_{j = 1}^i \left( \frac{1}{d} a_{ij} \nabla_x \cdot \langle \mathcal{B}(v) g^{(j)} \rangle  \right) .
\end{equation}
 In the limit $\varepsilon \to 0$ by inserting \eqref{InternalStagesForg} into \eqref{Tstar} and using orthogonality of $\A(v)$ and $\B(v)$, we obtain
\begin{equation}
\label{eq:T_star_stage}
\begin{split}
T_*^{(i)} = \, & T^n - \frac{\Delta t}{d} \sum_{j = 1}^{i-1} a_{ij} \nabla_x \cdot \langle \mathcal{B}(v) g^{(j)} \rangle \\
& + \frac{\Delta t}{d} \sum_{j=1}^{i-1} \tilde{a}_{ij} \nabla_x \cdot \left[ \frac{\tau}{2} \langle \mathcal{B}(v) \otimes \mathcal{B}(v) \rangle \nabla_x T^{(j)} - (d+2)(u^{(j)}T^{(j)}) \right] \\
& + \frac{\Delta t}{d} \sum_{j = 1}^{i-1} a_{ij} \nabla_x \cdot \langle \mathcal{B}(v) g^{(j)} \rangle.
\end{split}
\end{equation}
By using the following properties 
\begin{equation}
\begin{split}
\langle \mathcal{B}(v) \otimes \mathcal{B}(v) \rangle = 2(d+2)I , \;\; 
\langle \mathcal{B}(v) Q \rangle = \frac{d+2}{\tau} (uT)
\end{split}
\end{equation}
we get
\begin{equation}\begin{split} 
T_*^{(i)} &= T^n + \Delta t \frac{d+2}{d} \sum_{j=1}^{i-1} \tilde{a}_{ij} \left( \tau \Delta_x T^{(j)} - \nabla_x \cdot (u^{(j)}T^{(j)}) \right).
\end{split}
\end{equation}
Moreover, 
for $\epsilon \rightarrow 0$ we have 
\begin{align*}
T^{(i)} - T^n &= T^{(i)} - T_*^{(i)} + T_*^{(i)} - T^n \\
&= \left( -\frac{2}{d+2} + 1 \right) (T_*^{(i)} - T^n) \\
&= \Delta t \sum_{j=1}^{i-1} \tilde{a}_{ij} \left[ \tau \Delta_x T^{(j)} - \nabla_x \cdot (u^{(j)} T^{(j)}) \right].
\end{align*}

due to
\begin{eqnarray*}
	T^{(i)}-T_*^{(i)} 	
	= -\frac{2}{d+2}(T_*^{(i)}  -T^n )
\end{eqnarray*}
as previously and, altogether
\begin{equation}\begin{split} \label{FinalForT}
		T^{(i)} &= T^n + \Delta t\sum_{j=1}^{i-1} \tilde{a}_{ij} \left( \tau \Delta_x T^{(j)} - \nabla_x \cdot (u^{(j)}T^{(j)}) \right).
	\end{split}
\end{equation}
Since we are considering a GSA IMEX Runge Kutta scheme, the last internal stages of \eqref{FinalForu}, \eqref{FinalForp} and \eqref{FinalForT} form a consistent high order Implicit-Explicit time discretization of the limit equations.
\end{proof}

\begin{theorem}
	If the IMEX method is of type CK and satisfies the GSA property and the initial velocity field is divergence free, then in the fluid dynamic limit $\varepsilon \rightarrow 0$ the IMEX RK scheme \eqref{InternalStages} becomes a consistent discretization for the incompressible Navier Stokes equations.
\end{theorem}
The proof mainly follows the same line of reasoning as in Theorem~\ref{ThmFullBGKTypeA}. At \eqref{SecondProofReference}, we have to use that the initial velocity field is divergence free in order to get divergence free intermediate velocity fields, since $a_{11}= 0$ for IMEX methods of type CK.

The algorithm for solving the system of equations \eqref{eq:moments2} will contain the above discussed class of GSA IMEX Runge Kutta schemes and is summarized in Algorithm \ref{IMEX}.

\begin{algorithm}\caption{An Asymptotic Preserving IMEX Runge-Kutta scheme for the low Mach limit} \label{IMEX}\scriptsize
\begin{enumerate}
\item Compute $g^{(i)}$ by solving
\begin{equation}
\label{FirstStep}
\begin{split}
g^{(i)} = \, &\frac{\varepsilon^2 \tau}{\varepsilon^2 \tau + a_{ii} \Delta t} g^n \\
&- \frac{\Delta t \tau}{\varepsilon^2 \tau + a_{ii} \Delta t} \sum_{j=1}^{i-1} \tilde{a}_{ij} \Biggl( \nabla_x \cdot \left[ \mathcal{A}(v) u^{(j)} + \frac{\mathcal{B}(v)}{2} T^{(j)} + \varepsilon (I-P)(vg^{(j)}) \right] - Q^{(j)} \Biggr) \\
&- \frac{\Delta t}{\varepsilon^2 \tau + a_{ii} \Delta t} \sum_{j=1}^{i-1} a_{ij} g^{(j)}.
\end{split}
\end{equation}

\item Compute ${\tilde p^{(i)}}$ by solving the following Helmholtz equation
\begin{equation}
\label{SecondStep}
\begin{split}
\frac{d+2}{d}  a_{ii}^2\Delta_x \tilde{p}^{(i)} - \frac{\varepsilon^2}{\Delta t^2} \tilde{p}^{(i)} = \, & -\frac{\varepsilon}{\Delta t^2 } p^{(i)*} \\
& + \frac{d+2}{d \Delta t } \sum_{j = 1}^{i-1} a_{ij} \nabla_x \cdot u^{(j)} + \frac{(d+2) a_{ii}}{d \Delta t } \nabla_x \cdot u^{(i)*} \\
& - \frac{(d+2) a_{ii}}{d } \sum_{j=1}^{i-1} a_{ij} \Delta_x \tilde{p}^{(j)}.
\end{split}
\end{equation}

with
\begin{equation}
\begin{split}
{p^{(i)}}^* &= p^n - \Delta t  \sum_{j = 1}^i \left( a_{ij} \frac{1}{d} \nabla_x \cdot \langle \mathcal{B}(v) g^{(j)} \rangle \right), \\
{u^{(i)}}^* &= u^n - \Delta t  \sum_{j = 1}^i \left( a_{ij} \nabla_x \cdot \langle \mathcal{A}(v) g^{(j)} \rangle \right).
\end{split}
\end{equation}
\item Compute $u^{(i)}$ by solving
\begin{equation}
\begin{split}\label{ThirdStep}
{u^{(i)}} = u^n - \Delta t  \sum_{j = 1}^i \left( a_{ij} \nabla_x \cdot \langle \mathcal{A}(v) g^{(j)} \rangle \right) -\Delta t \sum_{j=1}^{i} a_{ij} \nabla_x \tilde p^{(j)}.
\end{split}
\end{equation}
\item Compute $\rho^{(i)}$ by solving
\begin{equation}
\begin{split}\label{FifthStep}
{\rho^{(i)}} =   \rho^n  - \frac{\Delta t}{\varepsilon} \sum_{j=1}^{i} a_{ij} \nabla_x \cdot u^{(j)} = \rho^n  + \frac{d}{d+2}
(\varepsilon \tilde p^{(i)} - {p^{(i)}}^* )\end{split}
\end{equation}
\item Compute $T^{(i)}$ by solving
\begin{equation}
\label{FourthStep}
\begin{split}
T^{(i)} &= T^n - \frac{\Delta t}{d} \sum_{j = 1}^i a_{ij} \nabla_x \cdot \langle \mathcal{B}(v) g^{(j)} \rangle - \frac{2 \Delta t}{d \varepsilon} \sum_{j=1}^{i} a_{ij} \nabla_x \cdot u^{(j)} \\
&= T^n + p^{(i)*} - p^n + \frac{2}{d} (\rho^{(i)} - \rho^n).
\end{split}
\end{equation}

\end{enumerate}
\label{alg1FullBGK}
\end{algorithm}

\subsection{Numerical treatment of the Collision Term}
The calculation of the second order perturbation term $Q$, defined in section two by the following relation
\[
Q = \frac{1}{\varepsilon^2 \tau} \left( \frac{M_F}{M} - 1 - \varepsilon Pf\right),
\]
is presenting a significant challenge. While the formula is analytically sound, its direct implementation in floating--point arithmetic is leading to severe cancellation errors as the system approaches the hydrodynamic limit $\varepsilon \to 0$. In order to ensure the robustness of the solver, a case distinction based on the magnitude of $\varepsilon$ is needed.

In this limit, the local Maxwellian $M_F$ approaches $M$. 
Specifically, $M_F$ can be approximated as $M$ plus a small perturbation in the direction of the hydrodynamic manifold, leading to the following expansion:
$$\frac{M_F}{M} = 1 + \varepsilon P_f + \mathcal{O}(\varepsilon^2)$$
Consequently, for very small values of $\varepsilon$, the ratio $\frac{M_F}{M}$ is nearly indistinguishable from $1 + \epsilon P_f$ in finite-precision arithmetic leading to severe cancellation errors, that are amplified by $\varepsilon^2$ in the denominator.

So, to circumvent this cancellation error, we employ a Taylor expansion of the ratio $\frac{M_F}{M_{101}}$ around $\varepsilon = 0$. Expanding up to $\mathcal{O}(\varepsilon^2)$ allows for the algebraic cancellation of the leading terms ($1$ and $\varepsilon P_f$), leaving $Q^{taylor}$ as the second-order coefficient ensuring that the calculation of $Q$ remains robust even as $\varepsilon$ approaches the hydrodynamic limit.

This means, for $\varepsilon < 10^{-4}$, we compute
\begin{equation}
\begin{split}
Q^{taylor} = \frac{1}{\tau} \biggl( & \frac{1}{2}(v \cdot u)^2 - \frac{1}{2d}|v|^2 |u|^2 + T(v \cdot u) \left( \frac{|v|^2 - (d+2)}{2} \right) \\
& + T^2 \left( \frac{|v|^4}{8} - \frac{2+d}{4}|v|^2 + \frac{d(d+2)}{8} \right) \biggr).
\end{split}
\end{equation}

\section{Discretization of Phase Space and Velocity Space}
\label{sec3b}
\subsection{Phase Space discretization}
In this section, we restrict ourselves to the case $d = 2$. In order to illustrate the basic principle of our spatial discretization and maintain readability, we describe a first-order finite difference scheme based on Algorithm~\ref{alg1FullBGK}. Our approach for the space discretization follows the methodology proposed in \cite{BOSCARINO}.
We introduce a mesh in the two-dimensional physical space with spacing $\Delta x$ and $\Delta y$, and define all variables at the cell centers of the resulting grid. 
\begin{itemize}
	\item The spatial discretization of the first order derivatives  $\nabla_x \cdot \langle \mathcal{A}(v) g \rangle$  in  \eqref{ThirdStep} and of $\nabla_x \cdot (I-P) (vg)$  in \eqref{FirstStep} of Algorithm~\ref{alg1FullBGK} are treated by local Lax Friedrichs flux functions. Thus, for $\nabla_x \cdot \langle \mathcal{A}(v) g \rangle$, we get 
\begin{equation}
\begin{split}
D_{LF} \cdot \langle \mathcal{A}(v) g \rangle = \, & \frac{1}{\Delta x} \left( \widehat{\langle \mathcal{A}(v) g \rangle}_{1_{i+1/2,j}} - \widehat{\langle \mathcal{A}(v) g \rangle}_{1_{i-1/2,j}} \right) \\
& + \frac{1}{\Delta y} \left( \widehat{\langle \mathcal{A}(v) g \rangle}_{2_{i,j+1/2}} - \widehat{\langle \mathcal{A}(v) g \rangle}_{2_{i,j-1/2}} \right).
\end{split}
\end{equation}
	with
	\begin{align}\nonumber
		\widehat{\langle \mathcal{A}(v) g \rangle}_{1_{i+1/2,j}} &= \frac{1}{2} \left( \langle \mathcal{A}(v) g \rangle_{1_{i,j}} + \langle \mathcal{A}(v) g \rangle_{1_{i+1,j}} - \alpha (u_{i+1,j} - u_{i,j})\right),\\ 
		\widehat{\langle \mathcal{A}(v) g \rangle}_{2_{i,j+1/2}} &= \frac{1}{2} \left( \langle \mathcal{A}(v) g \rangle_{2_{i,j}} + \langle \mathcal{A}(v) g \rangle_{2_{i,j+1}} - \alpha (u_{i,j+1} - u_{i,j})\right)
	\end{align}
	with $\alpha = 1$ and the subscripts $1$ and $2$ denoting the first and second column of the matrix. The space discretization for $\nabla_x \cdot (I-P)(vg)$ is done analogously.
		\item The remaining first order derivatives are discretized by central difference approximations, e.g. $D_0 \cdot \langle \mathcal{B}(v) g \rangle$ for $\nabla_x \cdot \langle \mathcal{B}(v) g \rangle$.
	\begin{equation}
\label{eq:central_diff_split}
\begin{split}
D_{0} \cdot \langle \mathcal{B}(v) g \rangle = \, &\frac{1}{2\Delta x} (\langle \mathcal{B}(v) g \rangle_{i+1,j} - \langle \mathcal{B}(v) g \rangle_{i-1,j}) \\
&+ \frac{1}{2\Delta y} (\langle \mathcal{B}(v) g \rangle_{i,j+1} - \langle \mathcal{B}(v) g \rangle_{i,j-1}).
\end{split}
\end{equation}
	\item 	The Laplacian operator $\Delta_x p$ in \eqref{SecondStep} is given by
	\begin{align}
		D_{\Delta} p = \frac{p_{i+1,j} - 2p_{i,j}+ p_{i-1,j}}{\Delta x^2} + 
		\frac{p_{i,j+1} - 2p_{i,j}+ p_{i,j-1}}{\Delta y^2}.
	\end{align}
\end{itemize}
The extension to high–order space discretization is achieved by replacing all local Lax–Friedrichs flux functions with high order shock capturing schemes based on WENO reconstruction \cite{Shu}, while all central difference approximations are modified accordingly to attain fourth order accuracy.
We omit the details for simplicity.
We note also that the choice of the spatial discretization is not restricted to the above described Finite-Volume approach. We have also implemented a generalized Finite Difference discretization for arbitrary  grids based on a least square determination of the 
derivatives yielding similar results.

\subsection{Discretization of Velocity Space}
The transition from a continuous kinetic description to a computationally tractable model necessitates the use of Discrete Velocity Methods (DVM). To maintain high algebraic accuracy in the evaluation of moments we employ Gauss--Hermite quadrature for the discretization of the integrals.

For a given number of discrete velocities, Gauss--Hermite quadrature provides the maximum degree of polynomial exactness, specifically up to degree $2N-1$.
 Since the Maxwell distribution is characterized by a Gaussian weight function, the Gauss--Hermite framework allows for the exact recovery of equilibrium moments at the discrete level, provided the quadrature order is sufficiently high.
  It minimizes the number of required discrete velocity points while preserving the essential symmetry and conservation properties of the kinetic system, effectively mitigating the "curse of dimensionality" inherent in velocity space integration \cite{PhysRevE.73.056702}.
  Remember that 
  \[
  \langle f \rangle = \int_{\mathbb{R}^d}f(v) M(v)\ dv,
  \]
  with
  \[
  M(v) = \frac{1}{(2\pi)^{d/2}} \exp \left( - \frac{|v|^2}{2}\right).
  \]
  Since Gauss--Hermite quadrature serves as a specialized numerical integration technique designed to evaluate integrals of the form 
\[
\int_{-\infty}^\infty e^{-x^2} f(x)\ dx,
\]
we obtain by transforming the nodes $v_i = \sqrt{2}\xi_i$
\[
\langle f \rangle \approx \sum_{i_{1}}^{N_1} \dots \sum_{i_{d}}^{N_d} \frac{1}{\pi^{d/2}} w_{i_1} \dots w_{i_d} f(\sqrt{2} \xi_1 , \dots , \sqrt{2} \xi_d ), 
\]
where the weights are calculated by using the formula 
\[
w_i = \frac{2^{n-1} n! \sqrt{\pi}}{n^2 [H_{n-1}(\xi_i) ]^2}
\]
and $H_n(x)$ denotes the $n$-th degree Hermite polynomial.

So, e.g. for $\langle \mathcal{A}(v) g \rangle$ we obtain in the case $d=2$
\[
\langle \mathcal{A}(v) g \rangle \approx \sum_{i}^{N}\sum_{j}^{N} \frac{w_i w_j}{\pi} \mathcal{A}(\sqrt{2} \xi_i,\sqrt{2} \xi_j) g(x,\sqrt{2} \xi_i,\sqrt{2} \xi_j,t)
\]
with $x \in D \subset \mathbb{R}^2$ and $t \in [0,\infty)$ and $N$ denotes the number of grid points used to discretize velocity space. 
\section{Numerical tests}\label{sec4}
In this section, we present a set of numerical experiments designed to assess the performance of the proposed schemes. The tests focus on classical benchmark problems in one and two dimensions and aim to verify accuracy, stability, and robustness across different regimes.
\subsection{IMEX-RK schemes}\label{IMEXtabs}
For convergence studies, the following second and third order IMEX-RK methods are used: 
\begin{itemize}
    \item \textbf{Second-order GSA IMEX-RK scheme.}
    For the second-order scheme, we use the following Butcher tableaus:
    \begin{equation}
    \label{SecondOrderIMEX}
    \text{Explicit: } 
    \begin{array}{c|ccc}
        0 & 0 & 0 & 0 \\
        c & c & 0 & 0 \\
        1 & 1-1/(2c) & 1/(2c) & 0 \\ \hline
          & 1-1/(2c) & 1/(2c) & 0
    \end{array} \quad
    \text{Implicit: }
    \begin{array}{c|ccc}
        0 & 0 & 0 & 0 \\
        c & 0 & c & 0 \\
        1 & 0 & 1-\gamma & \gamma \\ \hline
          & 0 & 1-\gamma & \gamma
    \end{array},
    \end{equation}
    with $\gamma = (c-1/2)/(c-1)$ and $c = 1-1/\sqrt{2}$; see \cite{Bosc}.

    \item \textbf{Third-order GSA IMEX-RK scheme.}
    The third-order scheme is defined by the following tableaus:
    \begin{equation}
    \label{ThirdOrderIMEX}
    \text{Explicit: } 
    \begin{array}{c|ccccc}
        0           & 0 & 0 & 0 & 0 & 0 \\ 
        1/2         & 1/2 & 0 & 0 & 0 & 0 \\
        2/3         & 11/18 & 1/18 & 0 & 0 & 0 \\
        1/2         & 5/6 & -5/6 & 1/2 & 0 & 0 \\
        1           & 1/4 & 7/4 & 3/4 & -7/4 & 0 \\ \hline
                    & 1/4 & 7/4 & 3/4 & -7/4 & 0
    \end{array}
    \end{equation}
    \begin{equation*}
    \text{Implicit: }
    \begin{array}{c|ccccc}
        0           & 0 & 0 & 0 & 0 & 0 \\
        1/2         & 0 & 1/2 & 0 & 0 & 0 \\ 
        2/3         & 0 & 1/6 & 1/2 & 0 & 0 \\
        1/2         & 0 & -1/2 & 1/2 & 1/2 & 0 \\
        1           & 0 & 3/2 & -3/2 & 1/2 & 1/2 \\ \hline
                    & 0 & 3/2 & -3/2 & 1/2 & 1/2
    \end{array}.
    \end{equation*}
\end{itemize}
together with second, third and fifth order WENO reconstructions and central difference methods for the space discretization. 

\subsection{The direct numerical method for the  BGK equation}

For comparison of the results we use a direct numerical scheme for 
the BGK equation (\ref{eq:scaled}) explained in detail in \cite{TKR}. 
It is based on an explicit treatment of the advection term and an implicit treatment of the collision terms and  a second order IMEX-WENO method for the spatial discretization. The  scaled advection term   requires in the present case a time step restriction of the order $\epsilon$.
It is chosen according to (\ref{CFL})
with CFL number $CFL = 0.5$. Here, $v_{max} $ is the maximal absolute value of the velocities used for the discretization of the velocity space in the BGK model.


\subsection{Accuracy, Convergence and CPU times  for 1-D Test Problems}
\subsubsection{Accuracy}
In order to check the accuracy of our scheme, we consider the following one-dimensional smooth periodic test problem. Let $x \in [0,1]$ and solve Algorithm~\ref{alg1FullBGK} augmented with the smooth initial data 
\[
\rho(x,0) = 0.5 + \sin(2\pi x)
\]
and $u(x,0) = 0$, $T(x,0) = 0$ and $p = (\rho + T) /\eps$ as well as periodic boundary conditions. In Tables~\eqref{tab:convergence_1Dutau0}--\eqref{tab:convergence_1Dptau0dot01}, we show the error norms of Algorithm \ref{alg1FullBGK}, discretized with IMEX RK2 \eqref{SecondOrderIMEX} as well as WENO3 and second order central differences, at the pre-shock time $0.1$ using $CFL = 0.75$. As a reference solution we used the solution obtained by the finest mesh of $1025$ grid points. As expected, the absolute errors for $p$ are larger than for the other quantities because of the $\mathcal{O}(1/\varepsilon)$ scaling. However, a relative error of only approximately $0.3 \%$ at $N=257$ confirms the high precision of the numerical procedure.
The results confirm that the method achieves the expected third order of convergence.

\begin{table}[htbp]
    \centering
    \small
    \begin{tabular}{@{} l c cc cc cc @{}}
        \toprule
        \multirow{2}{*}{\textbf{$\varepsilon$}} & \multirow{2}{*}{$N$} & \multicolumn{2}{c}{\textbf{$L^1$ Error}} & \multicolumn{2}{c}{\textbf{$L^2$ Error}} & \multicolumn{2}{c}{\textbf{$L^\infty$ Error}} \\
        \cmidrule(lr){3-4} \cmidrule(lr){5-6} \cmidrule(lr){7-8}
        & & Error & EOC & Error & EOC & Error & EOC \\
        \midrule
        \multirow{4}{*}{\textbf{1}} 
		& 65  & 3.161e-4 & -- & 3.832e-4 & -- & 7.864e-4 & -- \\
       & 129  & 7.490e-5 & 2.0775 & 8.558e-5 & 2.1627 & 1.428e-4 & 2.4605 \\
        & 257  & 1.906e-5 & 1.9744 & 2.020e-5 & 2.0828 & 2.941e-5 & 2.2803 \\
        & 513  & 3.991e-6 & 2.2556 & 4.216e-6 & 2.2606 & 5.489e-6 & 2.4218 \\
        
        \addlinespace[0.3cm]
        \multirow{4}{*}{\textbf{0.01}} 
		& 65  & 2.950e-2 & -- & 3.620e-2 & -- & 5.88e-2 & -- \\
       & 129  & 5.40e-3 & 2.4500 & 9.00e-3 & 2.0031 & 2.17e-2 & 1.4359 \\
        & 257  & 7.621e-4 & 2.8238 & 1.30e-3 & 2.7731 & 3.60e-3 & 2.6100 \\
        & 513  & 1.690e-4 & 2.1722 & 2.688e-4 & 2.2950 & 6.914e-4 & 2.3640 \\
        
        
        \bottomrule
    \end{tabular}
    \caption{Convergence analysis in 1D with different values for $\varepsilon$ and $\tau = 0$ at $t = 0.1$ showing variable $u$ and using IMEX RK2 with WENO3 and second order central difference approximations. EOC denotes the Experimental Order of Convergence.}
    \label{tab:convergence_1Dutau0}
\end{table}

\begin{table}[htbp]
    \centering
    \small
    \begin{tabular}{@{} l c cc cc cc @{}}
        \toprule
        \multirow{2}{*}{\textbf{$\varepsilon$}} & \multirow{2}{*}{$N$} & \multicolumn{2}{c}{\textbf{$L^1$ Error}} & \multicolumn{2}{c}{\textbf{$L^2$ Error}} & \multicolumn{2}{c}{\textbf{$L^\infty$ Error}} \\
        \cmidrule(lr){3-4} \cmidrule(lr){5-6} \cmidrule(lr){7-8}
        & & Error & EOC & Error & EOC & Error & EOC \\
        \midrule
        \multirow{4}{*}{\textbf{1}} 
		& 65  & 4.770e-4 & -- & 5.979e-4 & -- & 8.881e-4 & -- \\
       & 129  & 9.313e-5 & 2.3569 & 1.034e-4 & 2.5307 & 2.047e-4 & 2.1172 \\
        & 257  & 1.949e-5 & 2.2561 & 2.286e-5 & 2.1782 & 2.462e-5 & 3.0552 \\
        & 513  & 3.543e-6 & 2.4599 & 4.492e-6 & 2.3475 & 3.743e-6 & 2.7179 \\
        
        \addlinespace[0.3cm]
        \multirow{4}{*}{\textbf{0.01}} 
		& 65  & 15.807 & -- & 18.874 & -- & 14.769 & -- \\
       & 129  & 1.3916 & 3.5058 & 1.604 & 3.5566 & 1.170 & 3.6577 \\
        & 257  & 2.164e-1 & 2.6848 & 2.518e-1 & 2.6713 & 1.84e-1 & 2.6663 \\
        & 513  & 3.870e-2 & 2.4848 & 4.61e-2 & 2.4508 & 4.584e-2 & 2.0077 \\

        \bottomrule
    \end{tabular}
    \caption{Convergence analysis in 1D with different values for $\varepsilon$ and $\tau = 0$ at $t = 0.1$ showing variable $p$ and using IMEX RK2 with WENO3 and second order central difference approximations. EOC denotes the Experimental Order of Convergence. Since $p$ scales as $O(1/\varepsilon)$, absolute error magnitudes are predictably higher. However, a relative error of only $\approx 0.3 \%$ at $N=257$ confirms the high precision of the numerical procedure.}
    \label{tab:convergence_1Dptau0}
\end{table}

\begin{table}[htbp]
    \centering
    \small
    \begin{tabular}{@{} l c cc cc cc @{}}
        \toprule
        \multirow{2}{*}{\textbf{$\varepsilon$}} & \multirow{2}{*}{$N$} & \multicolumn{2}{c}{\textbf{$L^1$ Error}} & \multicolumn{2}{c}{\textbf{$L^2$ Error}} & \multicolumn{2}{c}{\textbf{$L^\infty$ Error}} \\
        \cmidrule(lr){3-4} \cmidrule(lr){5-6} \cmidrule(lr){7-8}
        & & Error & EOC & Error & EOC & Error & EOC \\
        \midrule
        \multirow{4}{*}{\textbf{1}} 
		& 65  & 2.790e-4 & -- & 3.235e-4 & -- & 5.315e-4 & -- \\
       & 129  & 7.342e-5 & 1.9260 & 8.440e-5 & 1.9387 & 1.369e-4 & 1.9566 \\
        & 257  & 1.842e-5 & 1.9946 & 2.074e-5 & 2.0248 & 3.242e-5 & 2.0782 \\
        & 513  & 3.795e-6 & 2.2791 & 4.261e-6 & 2.2829 & 6.396e-6 & 2.3420 \\
        
        \addlinespace[0.3cm]
        \multirow{4}{*}{\textbf{0.01}} 
		& 65  & 2.730e-2 & -- & 3.390e-2 & -- & 5.45e-2 & -- \\
       & 129  & 4.70e-3 & 2.5363 & 7.50e-3 & 2.1669 & 1.67e-2 & 1.7076 \\
        & 257  & 6.955e-4 & 2.7593 & 1.10e-3 & 2.7624 & 2.40e-3 & 2.8156 \\
        & 513  & 1.218e-4 & 2.5135 & 2.042e-4 & 2.4442 & 5.310e-4 & 2.1583 \\
        \bottomrule
    \end{tabular}
    \caption{Convergence analysis in 1D with different values for $\varepsilon$ and $\tau = 0.01$ at $t = 0.1$ showing variable $u$ and using IMEX RK2 with WENO3 and second order central difference approximations. EOC denotes the Experimental Order of Convergence.}
    \label{tab:convergence_1Dutau0dot01}
\end{table}
\begin{table}[htbp]
    \centering
    \small
    \begin{tabular}{@{} l c cc cc cc @{}}
        \toprule
        \multirow{2}{*}{\textbf{$\varepsilon$}} & \multirow{2}{*}{$N$} & \multicolumn{2}{c}{\textbf{$L^1$ Error}} & \multicolumn{2}{c}{\textbf{$L^2$ Error}} & \multicolumn{2}{c}{\textbf{$L^\infty$ Error}} \\
        \cmidrule(lr){3-4} \cmidrule(lr){5-6} \cmidrule(lr){7-8}
        & & Error & EOC & Error & EOC & Error & EOC \\
        \midrule
        \multirow{4}{*}{\textbf{1}} 
		& 65  & 3.986e-4 & -- & 4.727e-4 & -- & 1.00e-3 & -- \\
       & 129  & 8.199e-5 & 2.2816 & 9.422e-4 & 2.3270 & 2.326e-4 & 2.1734 \\
        & 257  & 1.582e-5 & 2.3732 & 1.774e-5 & 2.4087 & 2.412e-5 & 3.2693 \\
        & 513  & 2.915e-6 & 2.4405 & 3.383e-6 & 2.3910 & 3.547e-6 & 2.7660 \\
        \addlinespace[0.3cm]
        \multirow{4}{*}{\textbf{0.01}} 
		& 65  & 14.76 & -- & 17.47 & -- & 13.40 & -- \\
       & 129  & 1.3323 & 3.4697 & 1.5251 & 3.5180 & 1.1253 & 3.5739 \\
        & 257  & 2.016e-1 & 2.7243 & 2.256e-1 & 2.7569 & 1.771e-1 & 2.6679 \\
        & 513  & 3.490e-2 & 2.5323 & 3.91e-2 & 2.5288 & 2.93e-2 & 2.5953 \\
 \bottomrule
    \end{tabular}
    \caption{Convergence analysis in 1D with different values for $\varepsilon$ and $\tau = 0.01$ at $t = 0.1$ showing variable $p$ and using IMEX RK2 with WENO3 and second order central difference approximations. EOC denotes the Experimental Order of Convergence. Since $p$ scales as $O(1/\varepsilon)$, absolute error magnitudes are predictably higher. However, a relative error of only $\approx 0.3 \%$ at $N=257$ confirms the high precision of the numerical procedure.}
    \label{tab:convergence_1Dptau0dot01}
\end{table}
\subsubsection{Comparison with BGK solver}
We now consider the following test case in 1D. Let $x \in [0,1]$ and consider the following perturbed density profile
\begin{align}\label{InitialVelocity1}
	\rho(x,0) = \begin{cases} 1, & x \in [0,0.2] \\
							 1+\epsilon^2, & x \in (0.2,0.3] \\ 
							  1, & x \in (0.3,0.7] \\
							   1-\epsilon^2, & x \in (0.7,0.8] \\
							    1, & x \in (0.8,1] \\
							  \end{cases}
\end{align}
and $T(x,0) = 0$, $u(x,0) = 0$ and $p(x,0) = (\rho(x,0) + T(x,0))/\varepsilon$. This setup is often used to demonstrate how the numerical scheme handles sharp transitions in density while maintaining stability in the diffusion limit \cite{DegondTang2011,DimarcoLoubereVignal2017}. We discretize the spatial domain into $1001$ and the velocity domain into $50$ grid points and use the IMEX RK2 scheme as well as WENO3 and second order central differences in all computations. Figures~\eqref{fig:1d-1} -- \eqref{fig:1d-3} show that our solution is in good agreement with the reference solution, computed by a standard BGK solver, across various values of $\varepsilon$ and $\tau$.
\begin{figure}[H]
	\centering
	\begin{subfigure}{0.325\textwidth}
		\centering
		\includegraphics[width=\linewidth]{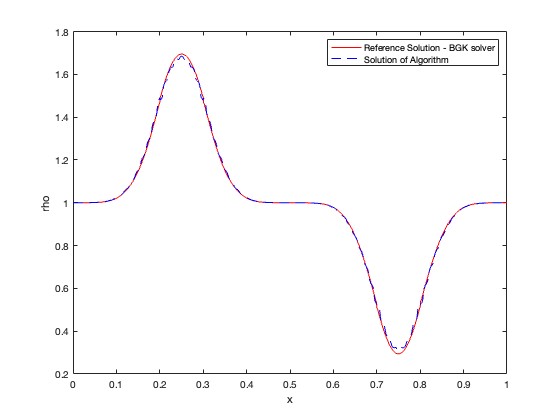}
	\end{subfigure}
	\centering
		\begin{subfigure}{0.325\textwidth}
		\centering
		\includegraphics[width=\linewidth]{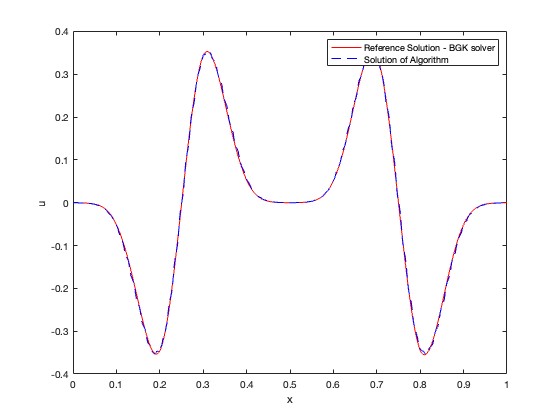}
	\end{subfigure}
	\centering
	\begin{subfigure}{0.325\textwidth}
		\centering
		\includegraphics[width=\linewidth]{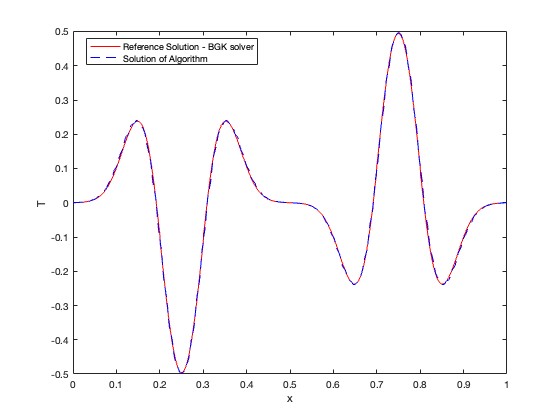}
	\end{subfigure}
	\centering
	\begin{subfigure}{0.325\textwidth}
		\centering
		\includegraphics[width=\linewidth]{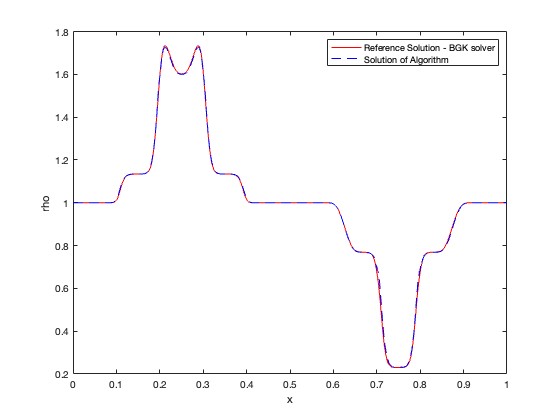}
	\end{subfigure}
		\centering
	\begin{subfigure}{0.325\textwidth}
		\centering
		\includegraphics[width=\linewidth]{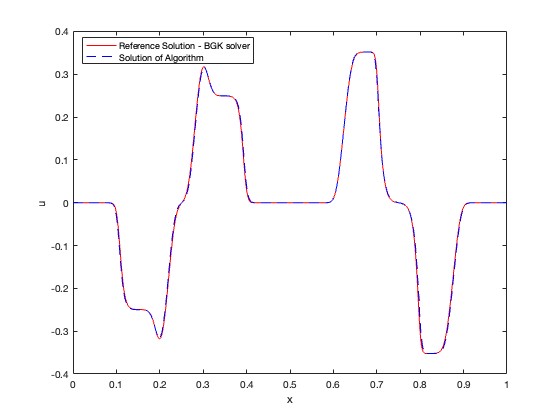}
	\end{subfigure}
	\centering
	\begin{subfigure}{0.325\textwidth}
		\centering
		\includegraphics[width=\linewidth]{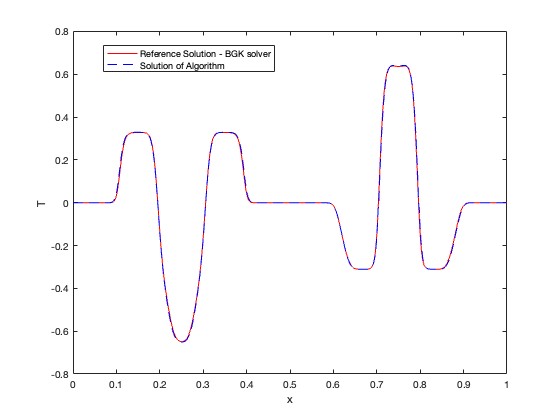}
	\end{subfigure}
	\caption{Comparison of solution to the test case in 1D for $\eps = 1$ and $\tau = 0.1$ (top row) and $\tau = 0.001$ (bottom row) at $t = 0.05$.}
	\label{fig:1d-1}
\end{figure}
\begin{figure}[H]
	\centering
	\begin{subfigure}{0.325\textwidth}
		\centering
		\includegraphics[width=\linewidth]{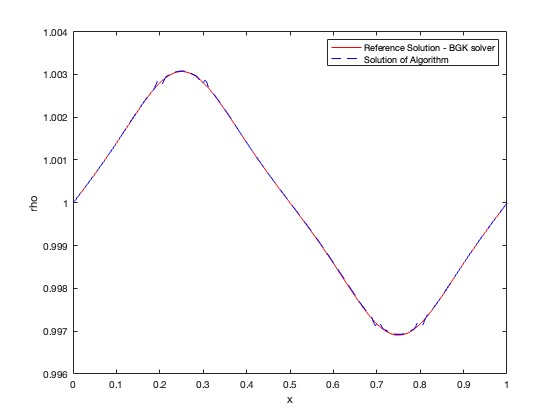}
	\end{subfigure}
	\centering
		\begin{subfigure}{0.325\textwidth}
		\centering
		\includegraphics[width=\linewidth]{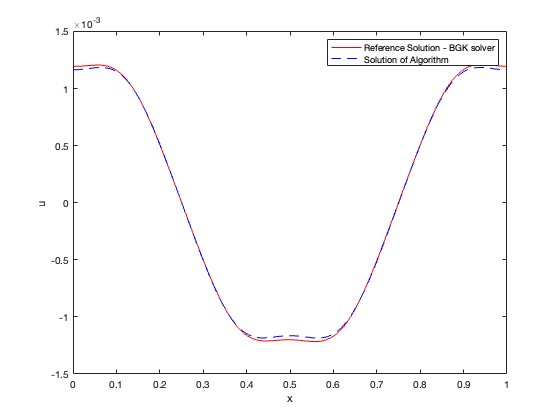}
	\end{subfigure}
	\centering
	\begin{subfigure}{0.325\textwidth}
		\centering
		\includegraphics[width=\linewidth]{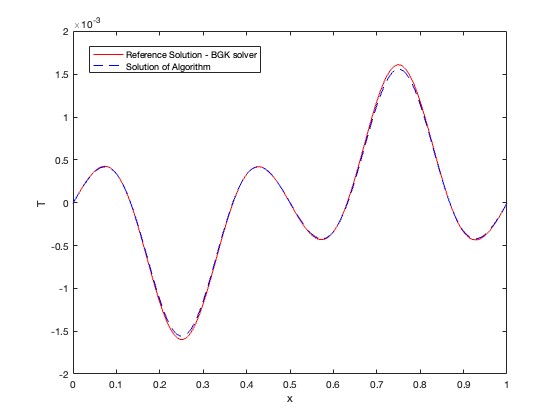}
	\end{subfigure}
	\centering
	\begin{subfigure}{0.325\textwidth}
		\centering
		\includegraphics[width=\linewidth]{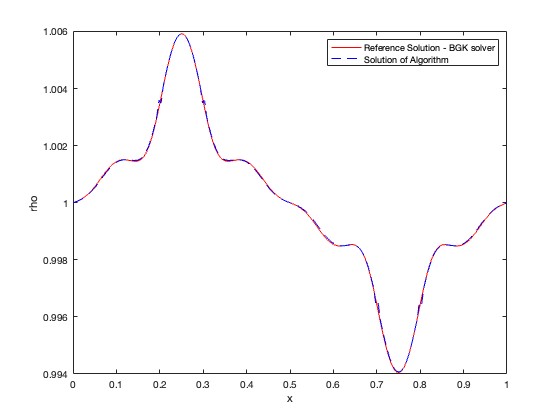}
	\end{subfigure}
		\centering
	\begin{subfigure}{0.325\textwidth}
		\centering
		\includegraphics[width=\linewidth]{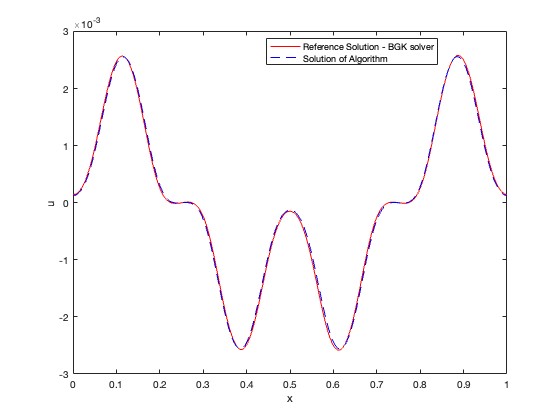}
	\end{subfigure}
		\centering
	\begin{subfigure}{0.325\textwidth}
		\centering
		\includegraphics[width=\linewidth]{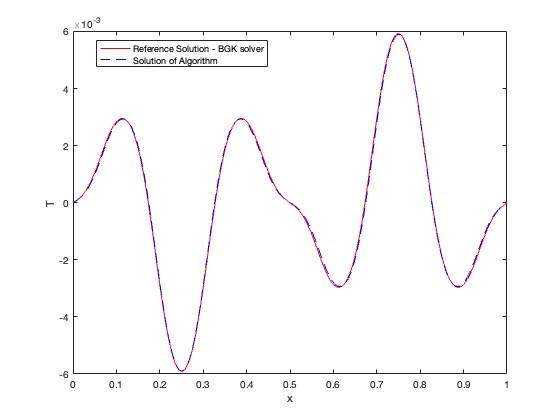}
	\end{subfigure}
	\caption{Comparison of solution to the test case in 1D for $\eps = 0.1$ and $\tau = 0.1$ (top row) and $\tau = 0.01$ (bottom row) at $t = 0.05$.}
	\label{fig:1d-2}
\end{figure}
\begin{figure}[H]
	\centering
	\begin{subfigure}{0.325\textwidth}
		\centering
		\includegraphics[width=\linewidth]{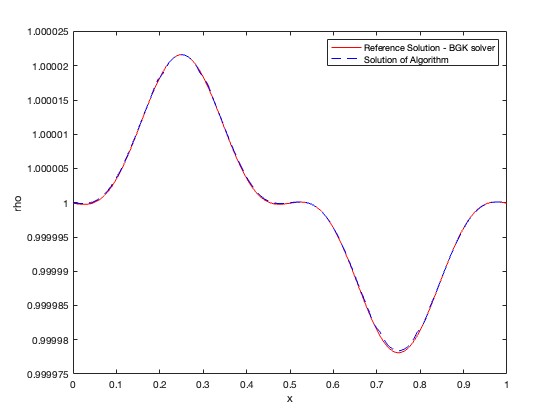}
	\end{subfigure}
	\centering
		\begin{subfigure}{0.325\textwidth}
		\centering
		\includegraphics[width=\linewidth]{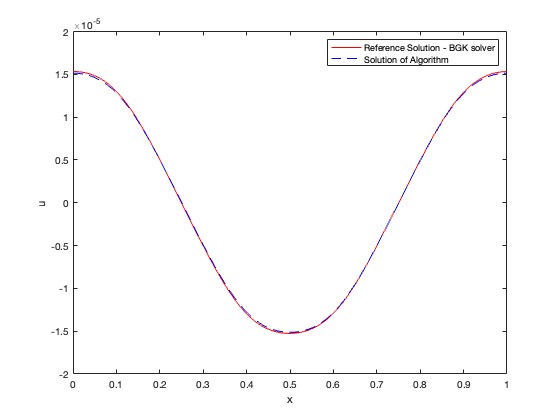}
	\end{subfigure}
	\centering
		\begin{subfigure}{0.325\textwidth}
		\centering
		\includegraphics[width=\linewidth]{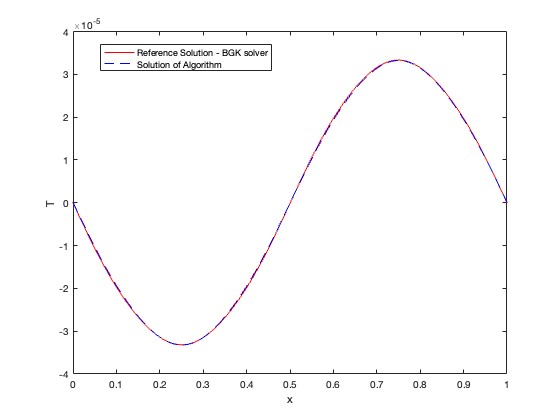}
	\end{subfigure}
		\centering
	\begin{subfigure}{0.325\textwidth}
		\centering
		\includegraphics[width=\linewidth]{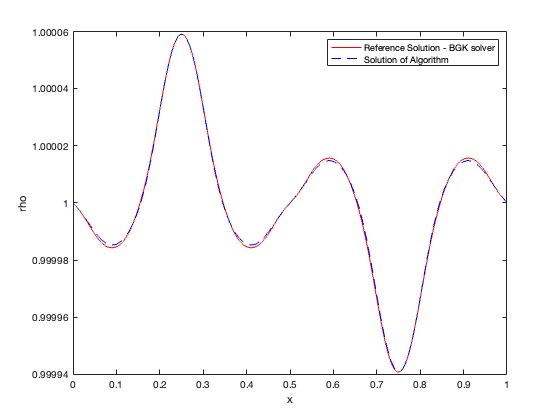}
	\end{subfigure}
		\centering
	\begin{subfigure}{0.325\textwidth}
		\centering
		\includegraphics[width=\linewidth]{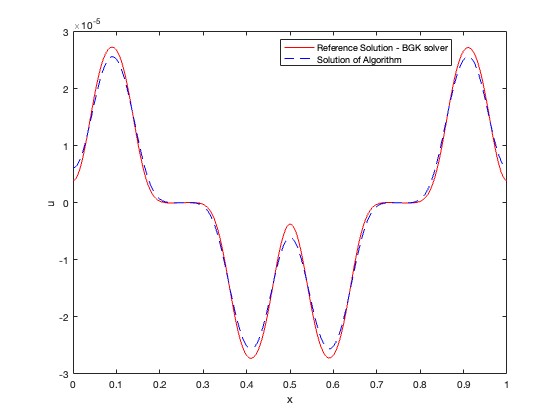}
	\end{subfigure}
	\centering
	\begin{subfigure}{0.325\textwidth}
		\centering
		\includegraphics[width=\linewidth]{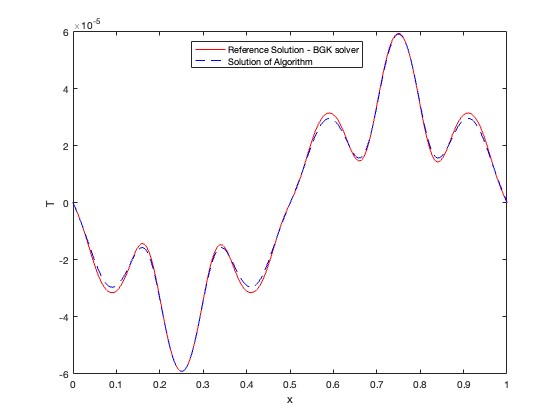}
	\end{subfigure}
	\caption{Comparison of solution to the test case in 1D for $\eps = 0.01$ and $\tau = 0.1$ (top row) and $\tau = 0.01$ (bottom row) at $t = 0.05$.}
	\label{fig:1d-3}
\end{figure}

\subsubsection{Comparison of CPU times}
	The CPU time for the BGK simulation obviously is proportional   to the value of $\epsilon$ according to the CFL condition
	(\ref{CFL}). 
For the AP method the time discretization is given as discussed in Remark \ref{remCFL}.
Using  equal spatial discretization size for both methods and, for the case  $\epsilon=1$ also equal time steps according to the above CFL-restrictions,  the CPU times for AP method  and standard BGK solver  for $\epsilon=1$  are comparable. 
	For  smaller values of $\epsilon$ the computation times for the BGK equation are becoming increasingly longer
	due to the $\epsilon$-dependent time step restriction (\ref{CFL}). The time step restriction for the AP-method 
	 (\ref{CFL2}) on the contrary is independent of $\epsilon$ and the time step can be chosen in the present situation equal to the one for $\epsilon=1$ at least for $\tau $ not too large.
	Changing for example  $\epsilon$ from $1$ to $0.1$ is equivalent to a rescaling of time by a factor $10$ and resulting in a CPU time which is 10 times longer for the BGK equation. 
	For small values of $\epsilon$  this leads to a considerable gain of computation times
	for the AP method  compared to a standard BGK simulation.  
	
\subsection{Thick shear layer test case in 2D}
Finally, we consider the following periodic two-dimensional model problem. Let $(x,y) \in \left[ 0,2 \pi\right]^2$ and let the initial flow to consist of a horizontal shear-layer of finite thickness, perturbed by a small amplitude vertical velocity of the form 
\begin{align}\label{InitialVelocity}
	u_1(x,y,0) = \begin{cases}\tanh(\frac{1}{\rho}(y-\pi/2)), & y \leq \pi \\ \tanh(\frac{1}{\rho}(3\pi/2 - y)), & y> \pi \end{cases}
\end{align}
and
\begin{align}
	u_2(x,y,0) = \delta \sin(x)
\end{align}
with the perturbation parameter $\delta = 0.05$ and the "thick" shear-layer width parameter $\rho = \frac{\pi}{15}$, see \cite{Bell} for details. 
A convergence study is shown in Tables~\eqref{tab:convergence_u1}--\eqref{tab:convergence_p} for the second and third order IMEX-RK methods (as described in Section~\ref{IMEXtabs}) together with third and fifth order WENO reconstructions as well as second and fourth order central difference approximations for the space discretization. 
To evaluate the convergence properties and accuracy of the numerical schemes, the global discretization error is measured using $L^1$, $L^2$ and $L^\infty$ norms. The reference solution is obtained from the numerical approximation computed on a fine mesh of $513 \times 513$ grid points. The results reported in Tables~\eqref{tab:convergence_u1}--\eqref{tab:convergence_p} confirm that the methods achieve the expected order of convergence.

For the numerical simulations, we employ a $128 \times 128$ spatial grid and discretize the velocity domain using $20 \times 20$ points.
We evaluate the performance of our method across three distinct regimes, i.e. rarefied, transitional and macroscopic limit regime. 

The rarefied and transitional regimes are used to validate the proposed method by comparing it against a standard BGK solver.
We consider cases with $\varepsilon = 1$ and $\varepsilon = 0.1$ in order to compare our solution with a reference solution computed by the BGK solver shortly described above, see also \cite{TKR}. Since the reference BGK solver is only second-order accurate, we utilize a comparable second-order version of our algorithm for this validation. Specifically, this version is obtained by employing second-order central differences combined with a minmod limiter for spatial reconstruction, rather than the Lax--Friedrichs flux functions and WENO reconstructions used in our higher order implementation.
As shown in Figures~\eqref{fig:2d-1} -- \eqref{fig:2d-3}, our solver provides a strong approximation in the rarefied  regime, matching the reference solution closely.

In the transitional regime, we further assess the validity of the new AP scheme by comparing it with the BGK solver, see Figure~\eqref{fig:2d-7}. While this comparison confirms again the basic validity of our approach, further strength lies in its high--order capabilities, that are able to successfully capture peak amplitudes and steep gradients that are smoothed out by the second--order approach, as can be nicely observed
in Figure~\eqref{fig:2d-8}. By utilizing IMEX RK2 with WENO3 and second order central difference approximations or IMEX RK3 with WENO5 and fourth order central difference approximations, we achieve sharper results, particularly as $\tau$ and $\varepsilon$ decrease.  
It is worth noting that the standard BGK simulation becomes computationally expensive in this regime, with computation times approximately ten times longer than in the $\varepsilon=1$ case.

Finally, we investigate the macroscopic limit regime using small values of $\varepsilon$ and varying values of $\tau$.
In this regime, standard BGK simulations are no longer computationally feasible. We again compare the higher order methods with each other. Figures~\eqref{fig:2d-9} -- \eqref{fig:2d-11} confirm that our solver also performs as expected in the incompressible regime for $\varepsilon = 10^{-6}$ and the results coincide well with those of incompressible Navier--Stokes simulation with different values of the diffusivity and with the incompressible Euler limit. As already mentioned, setting $\tau = 0.01$ yields the incompressible Navier--Stokes solution, while reducing $\tau$ to $0.0001$ causes Algorithm~\ref{alg1FullBGK} to approach the incompressible Euler case. 
\begin{table}[htbp]
    \centering
    \small
    \begin{tabular}{@{} l c cc cc cc @{}}
        \toprule
        \multirow{2}{*}{\textbf{Method}} & \multirow{2}{*}{$N$} & \multicolumn{2}{c}{\textbf{$L^1$ Error}} & \multicolumn{2}{c}{\textbf{$L^2$ Error}} & \multicolumn{2}{c}{\textbf{$L^\infty$ Error}} \\
        \cmidrule(lr){3-4} \cmidrule(lr){5-6} \cmidrule(lr){7-8}
        & & Error & EOC & Error & EOC & Error & EOC \\
        \midrule
        \multirow{4}{*}{\textbf{IMEX RK2}} 
        & 33  & 2.2198 & --   & 5.561e-1 & --   & 2.009e-1 & --   \\
       & 65  & 0.7009 & 1.6632 & 1.897e-1 & 1.5516 & 7.69e-2 & 1.3859 \\
       & 129  & 1.325e-1 & 2.4036 & 3.74e-2 & 2.3417 & 1.69e-2 & 2.1837 \\
        & 257  & 2.25e-2 & 2.5559 & 6.40e-3 & 2.5381 & 2.80e-3 & 2.5838 \\
        
        \addlinespace[0.3cm]
        \multirow{4}{*}{\textbf{IMEX RK3}} 
        & 33  & 1.0781 & --   & 2.843e-1 & --   & 1.10e-1 & --   \\
        & 65  & 9.18e-2 & 3.5533 & 2.62e-2 & 3.4411 & 1.28e-2 & 3.1052 \\
        & 129  & 1.06e-2 & 3.1136 & 4.00e-3 & 2.7207 & 3.10e-3 & 2.0251 \\
        & 257  & 5.40e-4 & 4.2963 & 2.21e-4 & 4.1670 & 1.689e-4 & 4.2162 \\
 \\
        
        
        \bottomrule
    \end{tabular}
    \caption{Convergence analysis for the double shear layer test case with $\varepsilon = 10^{-6}$ and $\tau = 0$ at $t = 1$ showing variable $u_1$ and using two different numerical methods, namely IMEX RK2 with WENO3 and IMEX RK3 with WENO5. EOC denotes the Experimental Order of Convergence.}
    \label{tab:convergence_u1}
\end{table}

\begin{table}[htbp]
    \centering
    \small
    \begin{tabular}{@{} l c cc cc cc @{}}
        \toprule
        \multirow{2}{*}{\textbf{Method}} & \multirow{2}{*}{$N$} & \multicolumn{2}{c}{\textbf{$L^1$ Error}} & \multicolumn{2}{c}{\textbf{$L^2$ Error}} & \multicolumn{2}{c}{\textbf{$L^\infty$ Error}} \\
        \cmidrule(lr){3-4} \cmidrule(lr){5-6} \cmidrule(lr){7-8}
        & & Error & EOC & Error & EOC & Error & EOC \\
        \midrule
        \multirow{4}{*}{\textbf{IMEX RK2}} 
        & 33  & 0.1288 & --   & 3.35e-2 & --   & 2.10e-2 & --   \\
       & 65  & 6.36e-2 & 1.0183 & 1.82e-2 & 0.8779 & 1.03e-2 & 1.0245 \\
       & 129  & 1.45e-2 & 2.1372 & 4.10e-3 & 2.1574 & 3.40e-3 & 1.6163 \\
        & 257  & 1.90e-3 & 2.8912 & 6.306e-4 & 2.6950 & 5.53e-4 & 2.6094 \\
        
        \addlinespace[0.3cm]
        \multirow{4}{*}{\textbf{IMEX RK3}} 
        & 33  & 0.1343 & --   & 3.17e-2 & --   & 1.75e-2 & --   \\
        & 65  & 2.43e-2 & 2.4679 & 6.40e-3 & 2.3101 & 3.50e-3 & 2.3128 \\
        & 129  & 9.13e-4 & 4.7325 & 2.932e-4 & 4.4469 & 2.30e-4 & 3.9328 \\
        & 257  & 2.167e-5 & 5.3967 & 7.531e-6 & 5.2829 & 6.857e-6 & 5.0683 \\
 \\
        
        
        \bottomrule
    \end{tabular}
    \caption{Convergence analysis for the double shear layer test case with $\varepsilon = 10^{-6}$ and $\tau = 0$ at $t = 1$ showing variable $u_2$ and using two different numerical methods, namely IMEX RK2 with WENO3 and IMEX RK3 with WENO5. EOC denotes the Experimental Order of Convergence.}
        \label{tab:convergence_u2}
\end{table}

\begin{table}[htbp]
    \centering
    \small
    \begin{tabular}{@{} l c cc cc cc @{}}
        \toprule
        \multirow{2}{*}{\textbf{Method}} & \multirow{2}{*}{$N$} & \multicolumn{2}{c}{\textbf{$L^1$ Error}} & \multicolumn{2}{c}{\textbf{$L^2$ Error}} & \multicolumn{2}{c}{\textbf{$L^\infty$ Error}} \\
        \cmidrule(lr){3-4} \cmidrule(lr){5-6} \cmidrule(lr){7-8}
        & & Error & EOC & Error & EOC & Error & EOC \\
        \midrule
        \multirow{4}{*}{\textbf{IMEX RK2}} 
        & 33  & 1.9774 & --   & 3.453e-1 & --   & 7.61e-2 & --   \\
       & 65  & 1.5075 & 0.3914 & 3.775e-1 & -0.1286 & 2.284e-1 & -1.5856 \\
       & 129  & 3.75e-1 & 2.0071 & 1.157e-1 & 1.7063 & 9.03e-2 & 1.3389 \\
        & 257  & 6.38e-2 & 2.5543 & 1.93e-2 & 2.5841 & 1.67e-2 & 2.4371 \\
        
        \addlinespace[0.3cm]
        \multirow{4}{*}{\textbf{IMEX RK3}} 
        & 33  & 1.0717 & --   & 2.024e-1 & --   & 4.47e-2 & --   \\
        & 65  & 2.955e-1 & 1.8586 & 7.78e-2 & 1.3795 & 5.75e-2 & -0.3621 \\
        & 129  & 1.65e-2 & 4.1627 & 4.50e-3 & 4.0995 & 3.10e-3 & 4.2155 \\
        & 257  & 3.20e-3 & 2.3556 & 5.23e-4 & 3.1169 & 1.337e-4 & 4.5325 \\
 \\
        
        
        \bottomrule
    \end{tabular}
    \caption{Convergence analysis for the double shear layer test case with $\varepsilon = 10^{-6}$ and $\tau = 0$ at $t = 1$ showing variable $p$ and using two different numerical methods, namely IMEX RK2 with WENO3 and IMEX RK3 with WENO5. EOC denotes the Experimental Order of Convergence.}
        \label{tab:convergence_p}
\end{table}

\begin{figure}[H]
	\centering
	\begin{subfigure}{0.45\textwidth}
		\centering
		\includegraphics[width=\linewidth]{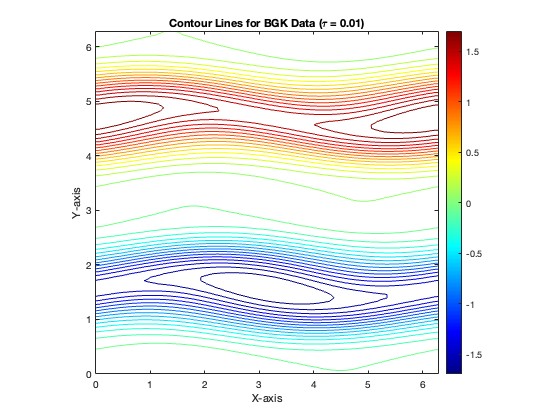}
	\end{subfigure}
	\centering
	\begin{subfigure}{0.45\textwidth}
		\centering
		\includegraphics[width=\linewidth]{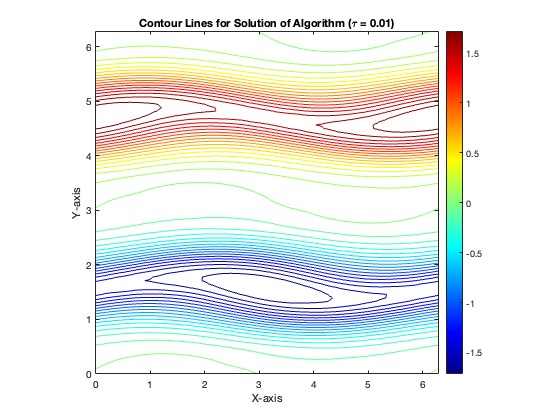}
	\end{subfigure}
	\caption{Solution of the shear layer test case for $\eps = 1$ and $\tau = 0.01$ at $t = 9$. Left: BGK \cite{TKR}, right: IMEX RK2, minmod }
	\label{fig:2d-1}
\end{figure}

\begin{figure}[H]
	\centering
	\begin{subfigure}{0.45\textwidth}
		\centering
		\includegraphics[width=\linewidth]{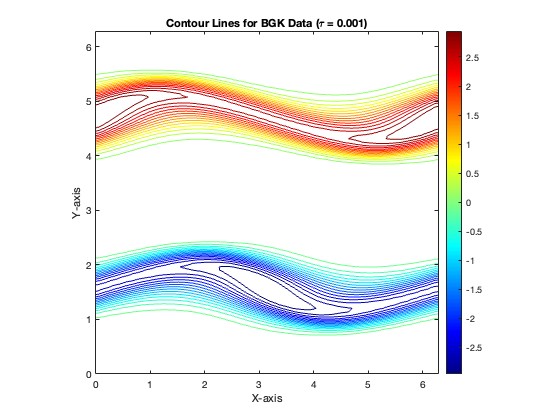}
	\end{subfigure}
	\centering
	\begin{subfigure}{0.45\textwidth}
		\centering
			\includegraphics[width=\linewidth]{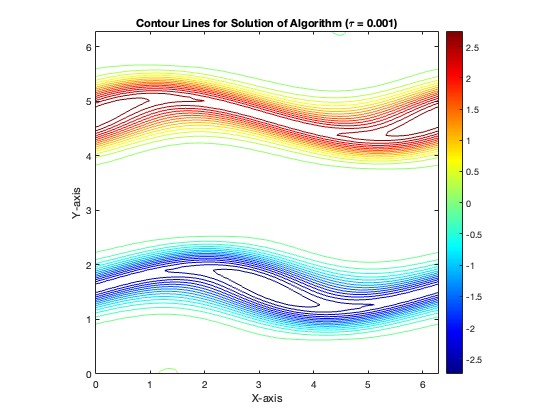}
	\end{subfigure}
	\caption{Solution of the shear layer test case for $\eps = 1$ and $\tau = 0.001$ at $t = 9$. Left: BGK \cite{TKR}, right: IMEX RK2, minmod}
	\label{fig:2d-2}
\end{figure}

\begin{figure}[H]
	\centering
	\begin{subfigure}{0.45\textwidth}
		\centering
		\includegraphics[width=\linewidth]{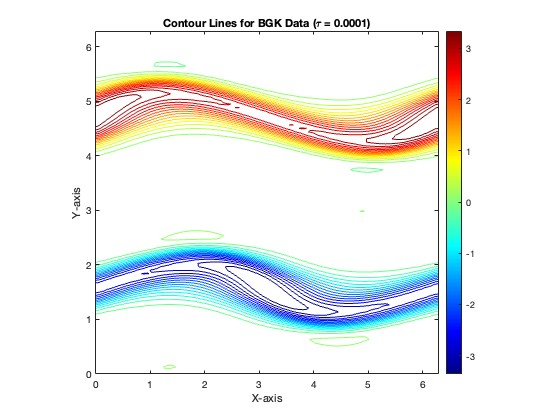}
	\end{subfigure}
	\centering
	\begin{subfigure}{0.45\textwidth}
		\centering
	\includegraphics[width=\linewidth]{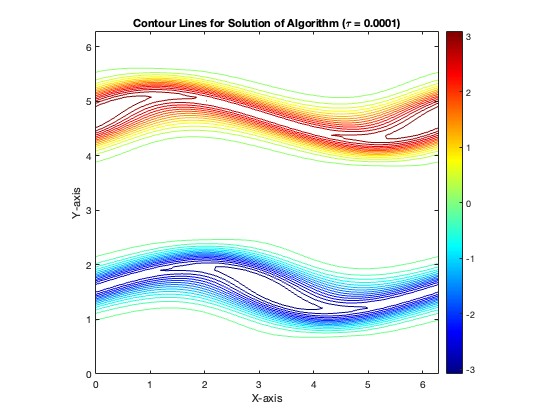}

	\end{subfigure}
	\caption{Solution of the shear layer test case for $\eps = 1$ and $\tau = 0.0001$ at $t = 9$. Left: BGK \cite{TKR}, right: IMEX RK2, minmod}
	\label{fig:2d-3}
\end{figure}

\begin{figure}[H]
	\centering
	\begin{subfigure}{0.45\textwidth}
		\centering
		\includegraphics[width=\linewidth]{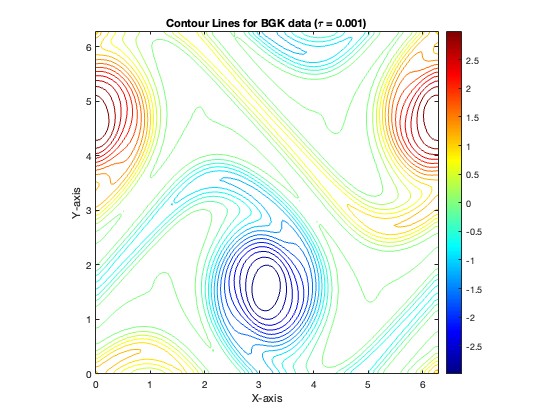}
	\end{subfigure}
	\centering
	\begin{subfigure}{0.45\textwidth}
		\centering
	\includegraphics[width=\linewidth]{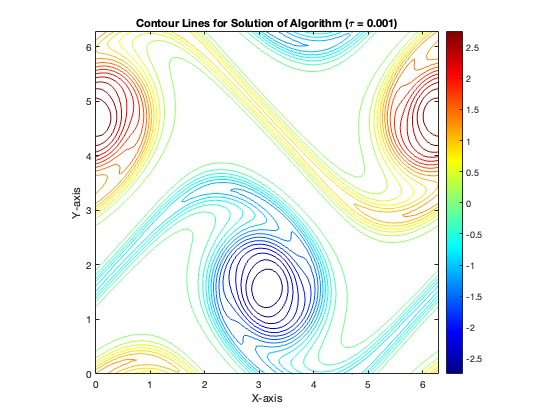}

	\end{subfigure}
	\caption{Solution of the shear layer test case for $\eps = 0.1$ and $\tau = 0.001$ at $t = 9$. Left: BGK \cite{TKR}, right: IMEX RK2, minmod}
	\label{fig:2d-7}
\end{figure}

\begin{figure}[H]
	\centering
	\begin{subfigure}{0.45\textwidth}
		\centering
		\includegraphics[width=\linewidth]{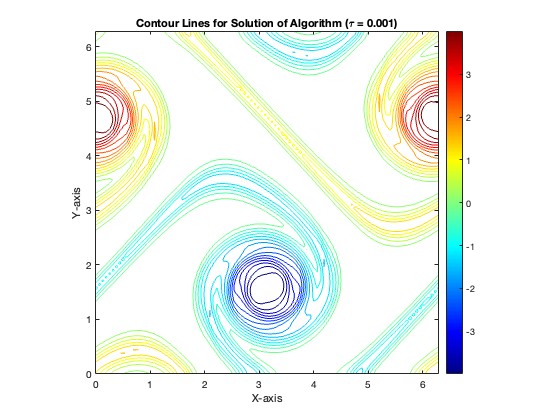}
	\end{subfigure}
	\centering
	\begin{subfigure}{0.45\textwidth}
		\centering
	\includegraphics[width=\linewidth]{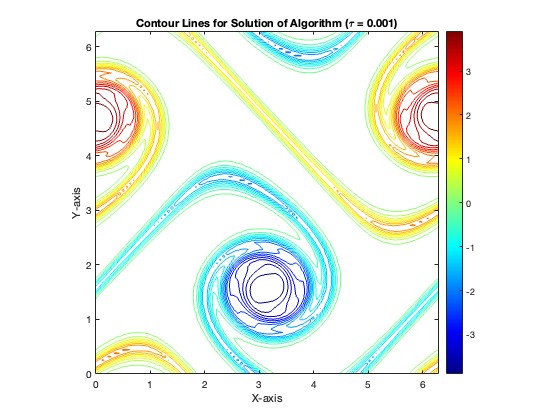}

	\end{subfigure}
	\caption{Solution of the shear layer test case for $\eps = 0.1$ and $\tau = 0.001$ at $t = 9$. (left: IMEX RK2, WENO3, right: IMEX RK3, WENO5)}
	\label{fig:2d-8}
\end{figure}

\begin{figure}[H]
	\centering
	\begin{subfigure}{0.45\textwidth}
		\centering
		\includegraphics[width=\linewidth]{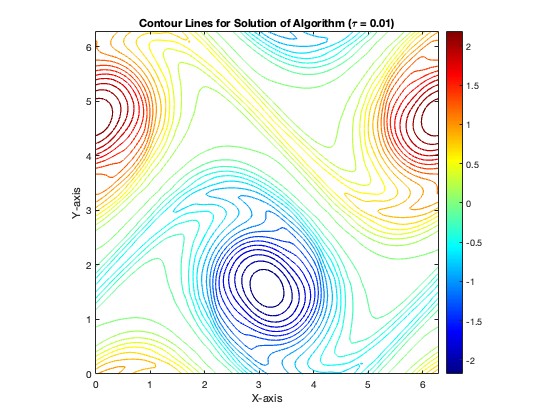}
	\end{subfigure}
	\centering
	\begin{subfigure}{0.45\textwidth}
		\centering
		\includegraphics[width=\linewidth]{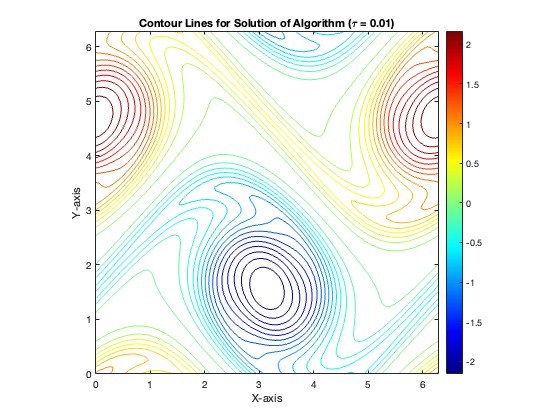}
	\end{subfigure}
	\caption{Solution of the shear layer test case for $\eps = 10^{-6}$ and $\tau = 0.01$  at $t = 9$. (left: IMEX RK2,WENO3,right: IMEX RK3,WENO5)}
	\label{fig:2d-9}
\end{figure}

\begin{figure}[H]
	\centering
	\begin{subfigure}{0.45\textwidth}
		\centering
		\includegraphics[width=\linewidth]{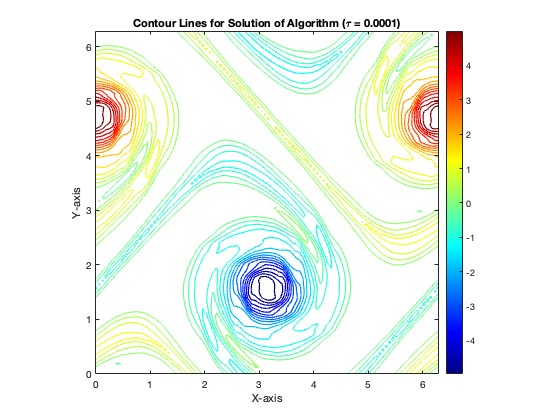}
	\end{subfigure}
	\centering
	\begin{subfigure}{0.45\textwidth}
		\centering
		\includegraphics[width=\linewidth]{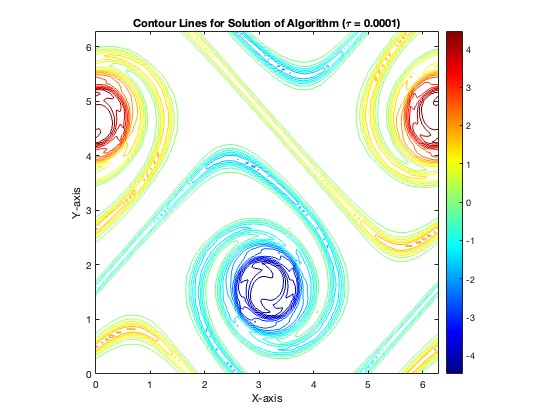}
	\end{subfigure}
	\caption{Solution of the shear layer test case for $\eps = 10^{-6}$ and  $\tau = 0.0001$ at $t = 9$. (left: IMEX RK2,WENO3,right: IMEX RK3,WENO5)}
	\label{fig:2d-10}
\end{figure}


\begin{figure}[H]
	\centering
	\begin{subfigure}{0.45\textwidth}
		\centering
		\includegraphics[width=\linewidth]{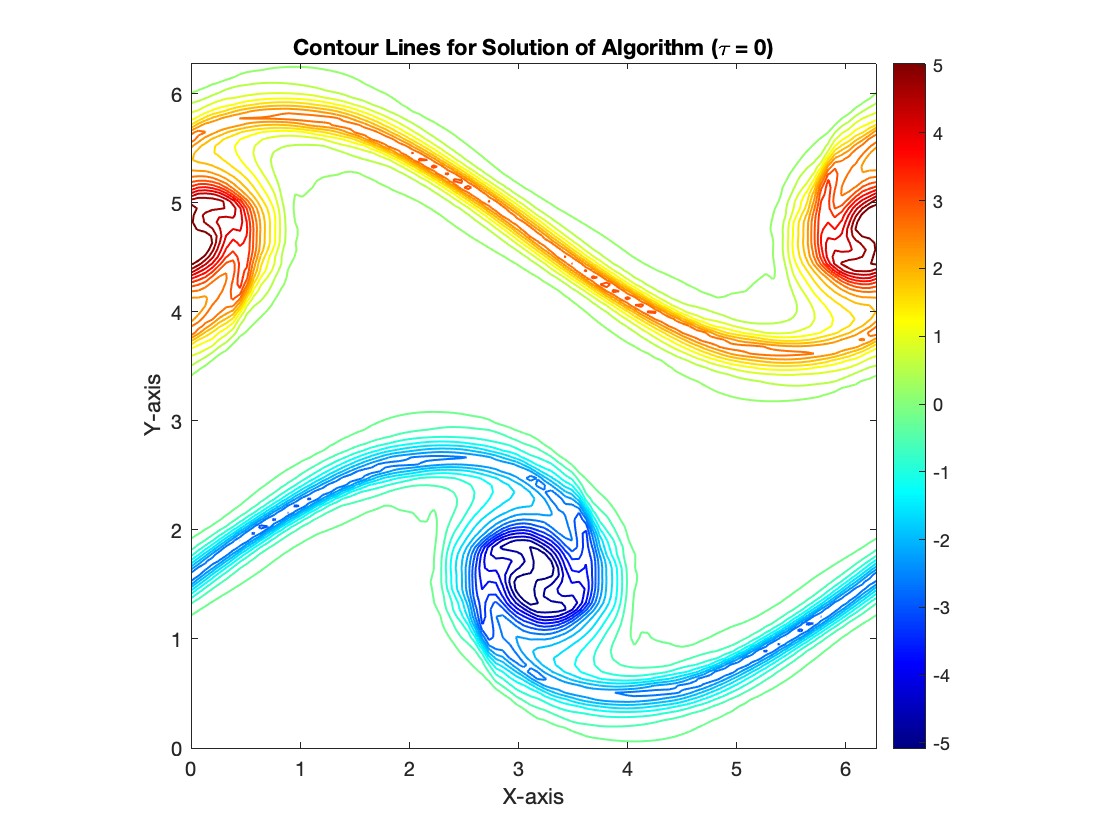}
	\end{subfigure}
	\centering
	\begin{subfigure}{0.45\textwidth}
	\centering
	\includegraphics[width=\linewidth]{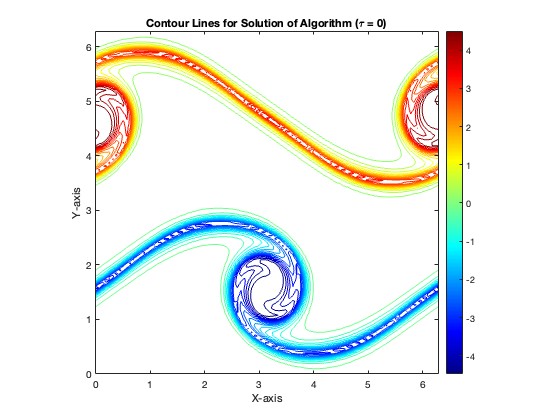}
	\end{subfigure}
	\caption{Solution of the shear layer test case for $\eps = 10^{-6}$ and $\tau = 0$ at $t = 6$.  (left: IMEX RK2,WENO3,right: IMEX RK3,WENO5)}
	\label{fig:2d-11}
\end{figure}
\section{Conclusion}\label{sec5}
In this work, we have introduced a novel high-order numerical framework for solving kinetic equations across a broad range of regimes, with a specific focus on the low Mach number limit. By employing a micro-macro decomposition, we reformulate the BGK equation into a coupled system that allows for a seamless transition from the rarefied kinetic scale to the incompressible hydrodynamic limit.

The proposed methodology leverages the strengths of Implicit--Explicit (IMEX) Runge Kutta schemes for temporal discretization, providing the necessary stability to handle the stiffness of the collision term and the fast acoustic modes inherent in low Mach number flows. Spatially, the combination of high-order WENO reconstructions and central difference approximations ensures that the scheme remains non--oscillatory and highly accurate. A central achievement of this framework is its Asymptotic-Preserving (AP) property; we have theoretically demonstrated and numerically verified that as the scaling parameter $\varepsilon$ tends to zero, the scheme degenerates into a consistent, high--order projection method for the incompressible Navier--Stokes equations.

Numerical experiments in one and two dimensions confirm the robustness of the solver. In 1D convergence studies, the method achieved its expected high-order accuracy, maintaining high precision even for smaller values of $\varepsilon$ where absolute pressure magnitudes are large. In the 2D "thick shear layer" benchmark, the higher--order versions of the algorithm demonstrated a superior ability to capture sharp gradients and peak amplitudes compared to standard second-order solvers, particularly in regimes with small relaxation times $\tau$. 

Future research will focus on several key areas. First, we intend to extend this framework beyond the periodic boundary conditions used in this study to include more physically complex constraints, such as inflow--outflow conditions and solid wall reflections. 
Additionally, we plan to extend the methodology to the full Boltzmann collision operator and explore the implementation on non--rectangular or irregular grids using high--order least-squares finite difference approximations, aimed at large--scale simulations of multiscale fluid phenomena.

\section*{Acknowledgements}  We thank Nicolas Crouseilles for interesting discussions and reference to unpublished closely related work.

\bibliographystyle{siamplain}

\bibliography{reference_last_SIAM}

@article {GSR,
	AUTHOR = {Golse, F. and Saint-Raymond, L.},
	TITLE = {The {N}avier-{S}tokes limit of the {B}oltzmann equation for
	bounded collision kernels},
	JOURNAL = {Invent. Math.},
	FJOURNAL = {Inventiones Mathematicae},
	VOLUME = {155},
	YEAR = {2004},
	NUMBER = {1},
	PAGES = {81--161},
}

@article{TKR,
	author = {Sudarshan Tiwari and Axel Klar and Giovanni Russo},
	title = {A meshfree arbitrary Lagrangian-Eulerian method for the BGK model of the Boltzmann equation with moving boundaries},
	journal = {Journal of Computational Physics},
	volume = {458},
	pages = {111088},
	year = {2022},
	issn = {0021-9991},
}

@article{SR,
	author    = {L. Saint-Raymond},
	title     = {From the {BGK} model to the {N}avier-{S}tokes equations},
	journal   = {Ann. Sci. Ecole Norm. Sup. },
	volume    = {36 (4)},
	pages     = {271--317},
	year      = {2003},
}

@article{JinPareschiToscani1998,
	author    = {Shi Jin and Lorenzo Pareschi and Giuseppe Toscani},
	title     = {Diffusive Relaxation Schemes for Multiscale Discrete‐Velocity Kinetic Equations},
	journal   = {SIAM Journal on Numerical Analysis},
	volume    = {35},
	number    = {6},
	pages     = {2405--2439},
	year      = {1998},
}

@article{AscherRuuthSpiteri1997,
	author    = {U.~M. Ascher and S.~J. Ruuth and R.~J. Spiteri},
	title     = {Implicit–{E}xplicit {R}unge–{K}utta methods for time‐dependent partial differential equations},
	journal   = {Applied Numerical Mathematics},
	volume    = {25},
	number    = {2--3},
	pages     = {151--167},
	year      = {1997},
}

@article{JinLevermore1996,
	author    = {Shi Jin and C.~David Levermore},
	title     = {Numerical methods for hyperbolic systems with stiff relaxation},
	journal   = {Communications on Pure and Applied Mathematics},
	volume    = {49},
	number    = {3},
	pages     = {323--343},
	year      = {1996}
}

@article{NaldiPareschiToscani2002,
	author    = {G. Naldi and Lorenzo Pareschi and Giuseppe Toscani},
	title     = {Numerical schemes for hyperbolic systems of conservation laws with stiff diffusive relaxation},
	journal   = {Applicable Analysis},
	volume    = {81},
	number    = {1},
	pages     = {1--26},
	year      = {2002}
}

@article {Levermore91,
	AUTHOR = {Bardos, Claude and Golse, Fran\c cois and Levermore, David},
	TITLE = {Fluid dynamic limits of kinetic equations. {I}. {F}ormal
	derivations},
	JOURNAL = {J. Statist. Phys.},
	FJOURNAL = {Journal of Statistical Physics},
	VOLUME = {63},
	YEAR = {1991},
	NUMBER = {1-2},
	PAGES = {323--344},
}

@book {Cer,
	AUTHOR = {Cercignani, Carlo},
	TITLE = {The {B}oltzmann equation and its applications},
	SERIES = {Applied Mathematical Sciences},
	VOLUME = {67},
	PUBLISHER = {Springer-Verlag, New York},
	YEAR = {1988},
	PAGES = {xii+455},
}

@article {Esposito,
	AUTHOR = {De Masi, A. and Esposito, R. and Lebowitz, J. L.},
	TITLE = {Incompressible {N}avier-{S}tokes and {E}uler limits of the
	{B}oltzmann equation},
	JOURNAL = {Comm. Pure Appl. Math.},
	FJOURNAL = {Communications on Pure and Applied Mathematics},
	VOLUME = {42},
	YEAR = {1989},
	NUMBER = {8},
	PAGES = {1189--1214},
}

@article {JinPareschiToscani1999,
	AUTHOR = {Jin, Shi and Pareschi, Lorenzo and Toscani, Giuseppe},
	TITLE = {Uniformly accurate diffusive relaxation schemes for multiscale
	transport equations},
	JOURNAL = {SIAM J. Numer. Anal.},
	FJOURNAL = {SIAM Journal on Numerical Analysis},
	VOLUME = {38},
	YEAR = {2000},
	NUMBER = {3},
	PAGES = {913--936},
}

@article {Dimarco_acta,
	AUTHOR = {Dimarco, G. and Pareschi, L.},
	TITLE = {Numerical methods for kinetic equations},
	JOURNAL = {Acta Num.},
	FJOURNAL = {Acta Numerica},
	VOLUME = {23},
	YEAR = {2014},
	PAGES = {369--520},
}

@article {Mieussens,
	AUTHOR = {Mieussens, Luc},
	TITLE = {Discrete velocity model and implicit scheme for the {BGK}
	equation of rarefied gas dynamics},
	JOURNAL = {Math. Models Methods Appl. Sci.},
	FJOURNAL = {Mathematical Models and Methods in Applied Sciences},
	VOLUME = {10},
	YEAR = {2000},
	NUMBER = {8},
	PAGES = {1121--1149},
}

@article {Illner,
	AUTHOR = {Platkowski, Tadeusz and Illner, Reinhard},
	TITLE = {Discrete velocity models of the {B}oltzmann equation: a survey
	on the mathematical aspects of the theory},
	JOURNAL = {SIAM Rev.},
	FJOURNAL = {SIAM Review. A Publication of the Society for Industrial and
	Applied Mathematics},
	VOLUME = {30},
	YEAR = {1988},
	NUMBER = {2},
	PAGES = {213--255},
}

@article {Mapundi,
	AUTHOR = {Banda, Mapundi and Klar, Axel and Pareschi, Lorenzo and
	Sea\"id, Mohammed},
	TITLE = {Lattice-{B}oltzmann type relaxation systems and high order
	relaxation schemes for the incompressible {N}avier-{S}tokes
	equations},
	JOURNAL = {Math. Comp.},
	FJOURNAL = {Mathematics of Computation},
	VOLUME = {77},
	YEAR = {2008},
	NUMBER = {262},
	PAGES = {943--965},
}

@article {DimarcoAP1,
	AUTHOR = {Dimarco, Giacomo and Pareschi, Lorenzo},
	TITLE = {Asymptotic preserving implicit-explicit {R}unge-{K}utta
	methods for nonlinear kinetic equations},
	JOURNAL = {SIAM J. Numer. Anal.},
	FJOURNAL = {SIAM Journal on Numerical Analysis},
	VOLUME = {51},
	YEAR = {2013},
	NUMBER = {2},
	PAGES = {1064--1087},
}

@article {DimarcoAP2,
	AUTHOR = {Dimarco, Giacomo and Pareschi, Lorenzo},
	TITLE = {Exponential {R}unge-{K}utta methods for stiff kinetic
	equations},
	JOURNAL = {SIAM J. Numer. Anal.},
	FJOURNAL = {SIAM Journal on Numerical Analysis},
	VOLUME = {49},
	YEAR = {2011},
	NUMBER = {5},
	PAGES = {2057--2077},
}

@article {KlarAP,
	AUTHOR = {Klar, Axel},
	TITLE = {An asymptotic-induced scheme for nonstationary transport
	equations in the diffusive limit},
	JOURNAL = {SIAM J. Numer. Anal.},
	FJOURNAL = {SIAM Journal on Numerical Analysis},
	VOLUME = {35},
	YEAR = {1998},
	NUMBER = {3},
	PAGES = {1073--1094},
}

@article {IMEX_PR,
	AUTHOR = {Pareschi, Lorenzo and Russo, Giovanni},
	TITLE = {Implicit-{E}xplicit {R}unge-{K}utta schemes and applications
	to hyperbolic systems with relaxation},
	JOURNAL = {J. Sci. Comput.},
	FJOURNAL = {Journal of Scientific Computing},
	VOLUME = {25},
	YEAR = {2005},
	NUMBER = {1-2},
	PAGES = {129--155},
}

@article {Bosc,
	AUTHOR = {Boscarino, S. and Pareschi, L. and Russo, G.},
	TITLE = {Implicit-{E}xplicit {R}unge-{K}utta schemes for hyperbolic
	systems and kinetic equations in the diffusion limit},
	JOURNAL = {SIAM J. Sci. Comput.},
	FJOURNAL = {SIAM Journal on Scientific Computing},
	VOLUME = {35},
	YEAR = {2013},
	NUMBER = {1},
	PAGES = {A22--A51},
}

@article {Carp,
	AUTHOR = {Kennedy, Christopher A. and Carpenter, Mark H.},
	TITLE = {Additive {R}unge-{K}utta schemes for
	convection-diffusion-reaction equations},
	JOURNAL = {Appl. Numer. Math.},
	FJOURNAL = {Applied Numerical Mathematics. An IMACS Journal},
	VOLUME = {44},
	YEAR = {2003},
	NUMBER = {1-2},
	PAGES = {139--181},
}

@article {Ruuth,
	AUTHOR = {Hundsdorfer, Willem and Ruuth, Steven J.},
	TITLE = {I{MEX} extensions of linear multistep methods with general
	monotonicity and boundedness properties},
	JOURNAL = {J. Comput. Phys.},
	FJOURNAL = {Journal of Computational Physics},
	VOLUME = {225},
	YEAR = {2007},
	NUMBER = {2},
	PAGES = {2016--2042},
}

@article {Multis,
	AUTHOR = {Dimarco, Giacomo and Pareschi, Lorenzo},
	TITLE = {Implicit-{E}xplicit linear multistep methods for stiff kinetic
	equations},
	JOURNAL = {SIAM J. Numer. Anal.},
	FJOURNAL = {SIAM Journal on Numerical Analysis},
	VOLUME = {55},
	YEAR = {2017},
	NUMBER = {2},
	PAGES = {664--690},
}

@article {Shu,
	AUTHOR = {Shu, Chi-Wang and Osher, Stanley},
	TITLE = {Efficient implementation of essentially nonoscillatory
	shock-capturing schemes},
	JOURNAL = {J. Comput. Phys.},
	FJOURNAL = {Journal of Computational Physics},
	VOLUME = {77},
	YEAR = {1988},
	NUMBER = {2},
	PAGES = {439--471},
}

@book {Chap,
	AUTHOR = {Chapman, Sydney and Cowling, T. G.},
	TITLE = {The {M}athematical {T}heory of {N}on-uniform {G}ases},
	PUBLISHER = {Cambridge University Press, Cambridge},
	YEAR = {1939},
	PAGES = {xxiii+404},
}

@article {Bosc1,
	AUTHOR = {Boscarino, Sebastiano and Pareschi, Lorenzo and Russo,
	Giovanni},
	TITLE = {A unified {IMEX} {R}unge-{K}utta approach for hyperbolic
	systems with multiscale relaxation},
	JOURNAL = {SIAM J. Numer. Anal.},
	FJOURNAL = {SIAM Journal on Numerical Analysis},
	VOLUME = {55},
	YEAR = {2017},
	NUMBER = {4},
	PAGES = {2085--2109},
}

@article {Klar99,
	AUTHOR = {Klar, Axel},
	TITLE = {Relaxation scheme for a {L}attice-{B}oltzmann-type discrete
	velocity model and numerical {N}avier-{S}tokes limit},
	JOURNAL = {J. Comput. Phys.},
	FJOURNAL = {Journal of Computational Physics},
	VOLUME = {148},
	YEAR = {1999},
	NUMBER = {2},
	PAGES = {416--432},
}

@article {JunkKlar,
	AUTHOR = {Junk, Michael and Klar, Axel},
	TITLE = {Discretizations for the incompressible {N}avier-{S}tokes
	equations based on the lattice {B}oltzmann method},
	JOURNAL = {SIAM J. Sci. Comput.},
	FJOURNAL = {SIAM Journal on Scientific Computing},
	VOLUME = {22},
	YEAR = {2000},
	NUMBER = {1},
	PAGES = {1--19},
}

@article {K992,
	AUTHOR = {Klar, Axel},
	TITLE = {An asymptotic preserving numerical scheme for kinetic
	equations in the low {M}ach number limit},
	JOURNAL = {SIAM J. Numer. Anal.},
	FJOURNAL = {SIAM Journal on Numerical Analysis},
	VOLUME = {36},
	YEAR = {1999},
	NUMBER = {5},
	PAGES = {1507--1527},
}

@article{BGK,
	title = {A Model for Collision Processes in Gases. I. Small Amplitude Processes in Charged and Neutral One-Component Systems},
	author = {Bhatnagar, P. L. and Gross, E. P. and Krook, M.},
	journal = {Phys. Rev.},
	volume = {94},
	issue = {3},
	pages = {511--525},
	year = {1954},
}

@book{BoscarinoPareschiRusso2024,
author = {Boscarino, Sebastiano and Pareschi, Lorenzo and Russo, Giovanni},
title = {Implicit-Explicit Methods for Evolutionary Partial Differential Equations},
publisher = {Society for Industrial and Applied Mathematics},
year = {2024},
address = {Philadelphia, PA},
edition   = {},
}

@article {Bosc2,
	AUTHOR = {Boscarino, Sebastiano and Russo, Giovanni},
	TITLE = {On a class of uniformly accurate {IMEX} {R}unge-{K}utta
	schemes and applications to hyperbolic systems with
	relaxation},
	JOURNAL = {SIAM J. Sci. Comput.},
	FJOURNAL = {SIAM Journal on Scientific Computing},
	VOLUME = {31},
	YEAR = {2009},
	NUMBER = {3},
	PAGES = {1926--1945},
	ISSN = {1064-8275,1095-7197},
}

@book {Hairer,
	AUTHOR = {Hairer, E. and Wanner, G.},
	TITLE = {Solving ordinary differential equations. {II}},
	SERIES = {Springer Series in Computational Mathematics},
	VOLUME = {14},
	EDITION = {Second},
	NOTE = {Stiff and differential-algebraic problems},
	PUBLISHER = {Springer-Verlag, Berlin},
	YEAR = {1996},
}

@article {BOSCARINO,
	AUTHOR = {Boscarino, Sebastiano and Qiu, Jing-Mei and Russo, Giovanni
	and Xiong, Tao},
	TITLE = {A high order semi-implicit {IMEX} {WENO} scheme for the
	all-{M}ach isentropic {E}uler system},
	JOURNAL = {J. Comput. Phys.},
	FJOURNAL = {Journal of Computational Physics},
	VOLUME = {392},
	YEAR = {2019},
	PAGES = {594--618},
}

@article {Bell,
	AUTHOR = {Bell, John B. and Colella, Phillip and Glaz, Harland M.},
	TITLE = {A second-order projection method for the incompressible
	{N}avier-{S}tokes equations},
	JOURNAL = {J. Comput. Phys.},
	FJOURNAL = {Journal of Computational Physics},
	VOLUME = {85},
	YEAR = {1989},
	NUMBER = {2},
	PAGES = {257--283},
}

@article {DegondTang2011,
    AUTHOR = {Degond, Pierre and Tang, Min},
     TITLE = {All speed scheme for the low {M}ach number limit of the
              isentropic {E}uler equations},
   JOURNAL = {Commun. Comput. Phys.},
  FJOURNAL = {Communications in Computational Physics},
    VOLUME = {10},
      YEAR = {2011},
    NUMBER = {1},
     PAGES = {1--31},
      ISSN = {1815-2406,1991-7120},
   MRCLASS = {76M20 (65M06 76L05 76N15)},
  MRNUMBER = {2775032},
MRREVIEWER = {Thomas\ H.\ Sonar},
}

@article {DimarcoLoubereVignal2017,
    AUTHOR = {Dimarco, Giacomo and Loub\`ere, Rapha\"el and Vignal,
              Marie-H\'el\`ene},
     TITLE = {Study of a new asymptotic preserving scheme for the {E}uler
              system in the low {M}ach number limit},
   JOURNAL = {SIAM J. Sci. Comput.},
  FJOURNAL = {SIAM Journal on Scientific Computing},
    VOLUME = {39},
      YEAR = {2017},
    NUMBER = {5},
     PAGES = {A2099--A2128},
      ISSN = {1064-8275,1095-7197},
   MRCLASS = {65M08 (65M12 76M12 76N99)},
  MRNUMBER = {3704264},
MRREVIEWER = {Mir\ Sajjad\ Hashemi},
}

@article{BennouneLemou,
	author = {Mounir Bennoune and Mohammed Lemou and Luc Mieussens},
title = {Uniformly stable numerical schemes for the Boltzmann equation preserving the compressible Navier–Stokes asymptotics},
journal = {Journal of Computational Physics},
volume = {227},
number = {8},
pages = {3781-3803},
year = {2008},
issn = {0021-9991},
}

@article{CrouseillesLemou,
  TITLE = {{An asymptotic preserving scheme based on a micro-macro decomposition for collisional Vlasov equations: diffusion and high-field scaling limits.}},
  AUTHOR = {Crouseilles, Nicolas and Lemou, Mohammed},
  JOURNAL = {{Kinetic and Related Models }},
  PUBLISHER = {{AIMS}},
  VOLUME = {4},
  NUMBER = {2},
  PAGES = {441-477},
  YEAR = {2011},
  HAL_VERSION = {v1},
}

@article{LemouMieussens,
author = {Lemou, Mohammed and Mieussens, Luc},
year = {2008},
month = {10},
pages = {334-368},
title = {A New Asymptotic Preserving Scheme Based on Micro-Macro Formulation for Linear Kinetic Equations in the Diffusion Limit},
volume = {31},
journal = {SIAM Journal on Scientific Computing},
}

@article{TATSIOS2025113500,
	author = {Tatsios, Giorgos and Chinnappan, Arun K. and Kamal, Arshad and Vasileiadis, Nikos and Docherty, Stephanie Y. and White, Craig and Gibelli, Livio and Borg, Matthew K. and Kermode, James R. and Lockerby, Duncan A.},
title = {A {DSMC-CFD} coupling method using surrogate modelling for low-speed rarefied gas flows},
journal = {Journal of Computational Physics},
volume = {520},
pages = {113500},
year = {2025},
issn = {0021-9991},
}

@article{PhysRevE.73.056702,
  title = {From the continuous to the lattice Boltzmann equation: The discretization problem and thermal models},
  author = {Philippi, Paulo C. and Hegele, Luiz A. and dos Santos, Lu\'{\i}s O. E. and Surmas, Rodrigo},
  journal = {Phys. Rev. E},
  volume = {73},
  issue = {5},
  pages = {056702},
  numpages = {12},
  year = {2006},
  month = {May},
  publisher = {American Physical Society},
}

\end{document}